\documentclass{amsart}

\usepackage[utf8]{inputenc}
\usepackage{amsmath,amsfonts,amssymb,diagrams,graphicx,color,mathrsfs,stmaryrd}
\usepackage{tikz,tikz-cd,tkz-euclide}
\usepackage{enumerate,array}
\usepackage{hyperref}

\newcolumntype{C}[1]{>{\centering\let\newline\\\arraybackslash\hspace{0pt}}m{#1}}

\newtheorem{theorem}[equation]{Theorem}
\newtheorem{lemma}[equation]{Lemma}
\newtheorem{claim}[equation]{Claim}
\newtheorem{corollary}[equation]{Corollary}

\newtheorem{remark}[equation]{Remark}
\newtheorem{example}[equation]{Example}
\newtheorem{setup}[equation]{Setup}

\renewcommand{\emph}{\textbf}

\def\df={\stackrel{\rm def}=}

\let\onto\twoheadrightarrow
\let\into\hookrightarrow
\let\sheaf\mathcal
\let\ideal\mathfrak
\let\tensor\otimes
\newcommand{\im}{\operatorname{im}}
\newcommand{\id}{\mathord{\text{id}}}

\newcommand{\unit}{{\scriptstyle\times}}
\newcommand{\fchar}{\operatorname{char}}
\newcommand{\divisor}{\operatorname{div}}
\newcommand{\cor}{\operatorname{cor}}
\newcommand{\res}{\operatorname{res}}
\newcommand{\ind}{\operatorname{ind}}

\newcommand{\inv}{\operatorname{inv}}
\newcommand{\isoto}{\stackrel{\approx}{\to}}

\newcommand{\Br}{\operatorname{Br}}
\newcommand{\Div}{\operatorname{Div}}
\newcommand{\Gal}{\operatorname{Gal}}
\newcommand{\Hom}{\operatorname{Hom}}

\newcommand{\Pic}{\operatorname{Pic}}
\newcommand{\Proj}{\operatorname{Proj}}
\newcommand{\Spec}{\operatorname{Spec}}
\newcommand{\Frac}{\operatorname{Frac}}

\newcommand{\FF}{\mathbb F}
\newcommand{\GG}{\mathbb G}

\newcommand{\NN}{\mathbb N}
\newcommand{\PP}{\mathbb P}
\newcommand{\QQ}{\mathbb Q}

\newcommand{\ZZ}{\mathbb Z}

\newcommand{\XX}{\mathscr{X}}

\begin{document}



\title{Cyclicity and indecomposability in the Brauer group of a $p$-adic curve}
\author{Eduardo Tengan}
\address{Universidade Federal de Santa Catarina\\
  Florianópolis, SC, Brazil}
\email{etengan@planetmail.com}
\keywords{Brauer group, $p$-adic curves, cyclic algebras}
\subjclass[2010]{16K (primary), 14G (secondary)}

\begin{abstract}
  For a $p$-adic curve $X$, we study conditions under which all
  classes in the $n$-torsion of $\Br(X)$ are $\ZZ/n$-cyclic.  We show
  that in general not all classes are $\ZZ/n$-cyclic classes.  On the
  other hand, if $X$ has good reduction and $n$ is prime to $p$, of if
  $X$ is an elliptic curve over $\QQ_p$ with split multiplicative
  reduction and $n$ is a power of $p$, then we prove that all order
  $n$ elements of $\Br(X)$ are $\ZZ/n$-cyclic.  Finally, if $X$ has
  good reduction and its function field $K(X)$ contains all $p^2$-th
  roots of $1$, we show the existence of indecomposable division
  algebras over $K(X)$ with period $p^2$ and index $p^3$.
\end{abstract}

\maketitle


\section{Introduction}

Let $p$ be a prime.  By a \emph{$p$-adic curve} $X$ we mean a smooth
geometrically connected projective curve over some finite extension
$K$ of $\QQ_p$.  Let $K(X)$ be the function field of $X$.  Division
algebras over $K(X)$ have been extensively studied by many different
authors (see for instance \cite{Saltman-p-adic},
\cite{Saltman-p-adic-c}, \cite{HHK}, \cite{BMT},
\cite{ParimalaSuresh}, \cite{BMT2}, among others).  In this paper, we
focus on two different aspects about $K(X)$-division algebras:
cyclicity and indecomposability.  We will be particularly interested
in classes of Azumaya algebras over $X$, classified by the Brauer
group of $X$.  As much as possible, we will try to tackle the
``thorny'' part of $\Br(K(X))$, namely its $p$-primary component, for
which fewer techniques are available, and much less is known about.

Let $F$ be any field, let $n \in \NN$ be prime to $\fchar F$, and
write $\Br(F)[n]$ for the $n$-torsion of the Brauer group $\Br(F)$ of
$F$.  A central simple $F$-algebra is called \emph{$\ZZ/n$-cyclic} if
it contains a maximal $\ZZ/n$-cyclic Galois étale subalgebra.  Then
its class in $\Br(F)[n]$ can be written as the cup product of an
element in $H^1(F, \mu_n) = F^\unit/n$ and a character in
$H^1(F, \ZZ/n)$; we call such classes $\ZZ/n$-cyclic as well.  Cyclic
algebras are the simplest amongst all central simple algebras, and it
is known that for a local or global field $F$ all classes in
$\Br(F)[n]$ are $\ZZ/n$-cyclic.

For the function field $K(X)$ of a $p$-adic curve $X$, it is known
that not all elements in $\Br(K(X))[n]$ are $\ZZ/n$-cyclic (see the
appendix of \cite{Saltman-p-adic}, by W.~Jacob and J.-P.~Tignol, for a
counter-example when $n=2$ and $p\ne 2$).  In \cite{BMT2} it was shown
that, for $n$ prime to $p$, any class in $\Br(K(X))[n]$ is the sum of
at most two $\ZZ/n$-cyclic classes.  That result led the author to
consider the question of whether all classes in the subgroup
$\Br(X)[n]$ of $\Br(K(X))[n]$ are $\ZZ/n$-cyclic or not.  It is not
difficult to show that if $X$ has good reduction (see
subsection~\ref{sec:integral-models} for the definition) and $n$ is
prime to $p$ then all classes in $\Br(X)[n]$ are $\ZZ/n$-cyclic; this
is done in theorem~\ref{thm:main}.  However, the answer to the above
question is negative in general: in
section~\ref{sec:non-cyclic-division}, for almost all
$p\equiv 1 \pmod 3$ we show the existence of a $p$-adic curve $X$ of
genus~10 having bad reduction and a class in $\Br(X)[3]$ that is not
$\ZZ/3$-cyclic.  Since we believe this construction to be of interest
on its own, we made an extra effort to identify and isolate the main
points that make it work, proving results in a bit more generality
than actually needed.

In the last two sections, we obtain results that also work for the
$p$-primary part of $\Br(K(X))$.  In section~\ref{sec:tate-curves}, we
study cyclicity for the Brauer group of an elliptic curve $E$ with
split multiplicative reduction.  Using Tate's $p$-adic uniformization
and Lichtembaum's duality, we show in theorem~\ref{thm:tatecurves}
that all classes in $\Br(E)[n]$ are $\ZZ/n$-cyclic if $n$ is either
prime or satisfies $\gcd(n, |\mu(K)|)=1$, where $\mu(K)$ is the group
of roots of unity in $K^\unit$.  As an interesting corollary, we get
that if $E$ is defined over $\QQ_p$ then all classes in $\Br(E)[p^r]$
are $\ZZ/p^r$-cyclic $(r\ge 1)$.

For the last result, recall that an $F$-division algebra is
\emph{indecomposable} if it cannot be expressed as the tensor product
of two nontrivial $F$-division algebras.  All division algebras of
equal period and index are indecomposable, while division algebras of
composite period are decomposable, so the problem of producing an
indecomposable division algebra is only interesting when the period
and index are unequal prime powers.  Albert constructed decomposable
division algebras of unequal (2-power) period and index in the 1930's,
but indecomposable division algebras of unequal period and index
appeared for the first time only in the late 1970's, in the papers
\cite{Saltman-Indecomposable} and \cite{ART}.  Since then there have
been several constructions, including \cite{Tignol}, \cite{JW},
\cite{J_2}, \cite{SvdB}, \cite{Karpenko}, \cite{Brussel6}, and
\cite{McKinnie}.  For a $p$-adic curve $X$ with good reduction,
indecomposable algebras over $K(X)$ were constructed in \cite{BMT} for
several combinations of (prime power) period/indices that are not
multiples of $p$.  In the last section, assuming $K$ contains all
$p^2$-th roots of unity, we construct indecomposable algebras over
$K(X)$ of period $p^2$ and index $p^3$ for any $p$-adic curve $X$ with
good reduction (theorem~\ref{thm:indec-algebr}).  Here, in contrast
with the tame case treated in \cite{BMT}, we do not have general
lifting theorems at our disposal, which unfortunately constrains our
methods to the above period $p^2$ and index $p^3$.  However, given the
paucity of results addressing the $p$-primary part of $\Br(K(X))$, we
still believe this particular construction to be of interest, as the
methods here employed may help future investigations of related
questions.


\section{Notation and setup}\label{sec:setup}

Throughout this paper, $p$ will be a fixed prime, and
\begin{itemize}
\item $K$ will be a $p$-adic field, i.e., a finite extension of the
  field of $p$-adic numbers $\QQ_p$;
  
\item $R$ will be the ring of integers of $K$, $\pi\in R$ will be a
  uniformizer, and $k = R/(\pi)$ will be its residue field (a finite
  field of characteritic $p>0$);

\item $\overline K$ will be an algebraic closure of $K$, and
  $G_K \df= \Gal(\overline K/K)$ will be its absolute Galois group;
  
\item $s = \Spec k$ and $\eta = \Spec K$ will be the closed and
  generic points of $\Spec R$;
  
\item $X$ will be a smooth geometrically connected projective curve
  over $\eta=\Spec K$.  We write $K(X)$ for its function field, and
  $X_0$ for the set of its closed points.
\end{itemize}
For each closed point $x\in X_0$, $\sheaf{O}_{X, x}$ is a discrete
valuation ring, and we denote by
\begin{itemize}
\item $v_x\colon K(X) \to \ZZ \cup \{\infty\}$ the associated discrete
  valuation;

\item $\kappa(x)$ its residue field (a $p$-adic field);

\item $\widehat{K(X)}_x$ the completion of $K(X)$ with respect to
  $v_x$.
\end{itemize}
We will still write
$v_x\colon \widehat{K(X)}_x \to \ZZ \cup \{\infty\}$ for the discrete
valuation induced by the original valuation $v_x$ on $K(X)$.

If $Z$ is a scheme over $\Spec A$ and $B$ is an $A$-algebra, we write
$Z(B)$ for the set of $(\Spec B)$-valued points of $Z$, and
$Z\tensor_A B$ for the fibered product $Z \times_{\Spec A} \Spec B$.
If $Z$ is defined over $\Spec R$, we denote by
$$Z_s\df= Z \times_{\Spec R} s = Z\tensor_R k
\qquad\text{and}\qquad Z_\eta\df= Z \times_{\Spec R} \eta = Z\tensor_R
K$$ its special and generic fibers, respectively.

For any abelian group $G$, we respectively write $G[n]$ and $G/n$ for
the kernel and cokernel of the multiplication-by-$n$ map
$G \stackrel{n}{\to} G$.  Moreover, if $G$ is a topological abelian
group, we denote by $G^\vee \df= \Hom_c(G, \QQ/\ZZ)$ its Pontryagin
dual, i.e., the group of all continuous morphisms
$\chi \colon G \to \QQ/\ZZ$ (here $\QQ/\ZZ$ is endowed with the
discrete topology).

Unless otherwise stated, all cohomology groups will either be Galois
or étale cohomology groups.  For any field $F$ and $n\in \NN$ not
divisible by $\fchar F$, we write $\mu_n$ for the group (or Galois
module or étale sheaf) of all $n$-th roots in the separable closure
$F^{\rm sep}$ of $F$.  We also write
$\delta_n \colon F^\unit \to H^1(F, \mu_n)$ for the connecting map
relative to the Kummer sequence
$1\to \mu_n \to F_{\rm sep}^\unit \stackrel{n}{\to} F_{\rm sep}^\unit
\to 1$.  If $n$ is clear from the context, we drop it and simply write
$\delta$ instead.


\section{Some auxiliary facts}

In this section, we recall some facts that will be used in the proofs
of the main results.  We advise the reader to skip this section,
coming back as needed.

\subsection{Ramification}

Let $\hat K$ be a discretely valued complete field.  Let
$v\colon \hat K \to \ZZ \cup \{\infty\}$ be the associated
(normalized) valuation, and $F$ be its residue field.  There is a
split exact sequence (\cite{SerreLF}, XII.\S3, corollary to
proposition~4, p.186)
\begin{center}
  \begin{tikzcd}
    0\arrow{r} & \Br(F) \arrow{r}{\inf} & \Br(\hat K)' \arrow{r}{\partial_v} & H^1(F, \QQ/\ZZ) \arrow{r}& 0
  \end{tikzcd}
\end{center}
where
$\Br(\hat K)' = \ker \bigl( \Br(\hat K)\to \Br(\hat K_{\rm nr})
\bigr)$ (here $\hat K_{\rm nr}$ denotes the maximal unramified
extension of $\hat K$).  In case $F$ is perfect
$\Br(\hat K)' = \Br(\hat K)$, and if $\fchar F\nmid n$ then
$\Br(\hat K)'[n] = \Br(\hat K)[n]$.  The map $\partial_v$ is called
\emph{ramification} or \emph{residue} map with respect to $v$.
Intuitively, $\partial_v$ can be thought of as ``valuation for
cohomology'', so that the inflation map identifies $\Br(F)$ with the
subgroup of ``unramified classes'' in $\Br(\hat K)'$.  By choosing a
uniformizer $\pi_v\in \hat K$, the map
$\chi\mapsto \delta_n(\pi_v) \cup \inf(\chi)$ yields a section of
$\partial_v$ for the $n$-torsion of the above sequence, and hence a
(non-canonical) isomorphism
\begin{align}
  \Br(F)[n] \oplus H^{1}(F, \ZZ/n\ZZ)&\isoto  \Br(\hat K)'[n]\nonumber\\
  (\alpha, \chi)&\mapsto \inf(\alpha) + \delta_n(\pi_v) \cup \inf(\chi)
\end{align}
which will be refered to as \emph{Witt decomposition} of
$\Br(\hat K)'[n]$.  By abuse of notation, we will usually omit the
inflation maps, writing $\alpha + \delta \pi_v \cup \chi$ instead.

It is fairly easy to compute the index of elements in
$\Br(\hat K)'[n]$, thanks to (see \cite{JW}, theorem~5.15, p.161 for a
proof)

\begin{lemma}[Nakayma's index formula]\label{lemma:NAK}
  In the above notation, let $n\in \NN$ and $\beta\in \Br(\hat K)'[n]$
  with Witt decomposition $\beta = \alpha + \delta\pi_v \cup \chi$.
  Write
  \begin{itemize}
  \item $|\chi|$ for the order of the character $\chi\in H^{1}(F, \ZZ/n\ZZ)$;
  \item $F(\chi)\supset F$ for the degree $|\chi|$ field extension defined by $\chi$;
  \item $\alpha|_{F(\chi)} \in \Br(F(\chi))$ for the restriction of
    $\alpha$ to $F(\chi)$.
  \end{itemize}
  Then the index of $\beta$ is given by
  $\ind \beta = |\chi| \cdot \ind \alpha|_{F(\chi)}$.
\end{lemma}

Next we define ramification/residue maps in more general contexts.
For $r\in \ZZ$ and $n\in \NN$ prime to $\fchar F$, denote the $r$-th
Tate twist of the Galois module (or étale sheaf) $\ZZ/n\ZZ$ by
$$\ZZ/n\ZZ(r) \df= \begin{cases}
  \mu_n^{\tensor r}& \text{if }r\ge 1\\
  \ZZ/n\ZZ& \text{if }r = 0\\
  \Hom(\mu_n^{\tensor (-r)}, \ZZ/n\ZZ)& \text{if }r\le -1
\end{cases}
$$
For $n\in \NN$ prime to $\fchar F$ and $i\ge 1$, there are
(noncanonical) split exact sequences (see \cite{GMS} II.7.9, p.18)
\begin{equation*}
  0 \to H^i(F, \ZZ/n\ZZ(r)) \stackrel{\inf}{\to} H^i(\hat K, \ZZ/n\ZZ(r)) \stackrel{\partial_v}{\to} H^{i-1}(F, \ZZ/n\ZZ(r-1)) \to 0
\end{equation*}
The map $\partial_v$ is the \emph{ramification} or \emph{residue} map
with finite coefficients.  When $i=2$ and $r=1$ it is compatible with
the previously defined ramification map via the isomorphism
$\Br(\hat K)[n] = H^2(\hat K, \mu_n)$.

Let $Z$ be an integral regular scheme with function field $K(Z)$.  Let
$D\subset Z$ be an irreducible Weil divisor with generic point
$\nu\in Z$.  Then $\sheaf{O}_{Z, \nu}$ is a discrete valuation ring;
write $v_D\colon K(Z) \to \ZZ \cup \{\infty\}$ and $\kappa(D)$ for the
corresponding valuation and residue field, respectively, and let
$\widehat{K(Z)}_D$ be the completion of $K(Z)$ with respect to $v_D$.
If $n\in \NN$ is prime to $\fchar \kappa(D)$, we define the
\emph{ramification} or \emph{residue} map
$\partial_D\colon H^i(K(Z), \ZZ/n(r))\to H^{i-1}(\kappa(D),
\ZZ/n(r-1))$ with respect to $D$ as the composition
\begin{center}
  \begin{tikzcd}
    H^i(K(Z), \ZZ/n(r)) \arrow{r}{\res}& H^i(\widehat{K(Z)}_D,
    \ZZ/n(r)) \arrow{r}{\partial_{v_D}}& H^{i-1}(\kappa(D),
    \ZZ/n(r-1))
  \end{tikzcd}
\end{center}

\begin{lemma}\label{lemma:Brauer}
  Let $Z$ be an integral regular scheme of dimension at most~$2$, and
  let $K(Z)$ be its function field.
  \begin{enumerate}[(i)]
  \item For any $n\in \NN$ invertible in $Z$, there is an exact
    sequence
    $$0 \longrightarrow \Br(Z)[n] \longrightarrow \Br(K(Z))[n]
    \stackrel{\bigoplus\partial_D}\longrightarrow \bigoplus_D H^1(\kappa(D), \QQ/\ZZ)[n]
    $$
    where $D$ runs over all prime divisors of $Z$.

  \item Suppose that $Z = \Spec A$ for some excellent $2$-dimensional
    local ring $A$.  Let $P\in Z$ be its closed point, and $\kappa(P)$
    be the residue field of $P$.  Suppose that $n\in \NN$ is prime to
    $\fchar \kappa(P)$.  Then the following sequence is a complex:
    $$\Br(K(Z))[n] \stackrel{\bigoplus\partial_D}\longrightarrow \bigoplus_D H^1(\kappa(D), \ZZ/n)
    \stackrel{\partial_P}\longrightarrow H^0(\kappa(P), \mu_n^{-1})
    $$
    Here $D$ runs over all prime divisors of $Z$, and
    $\partial_P = \sum_{Q} \cor_{\kappa(Q) / \kappa(P)} \circ
    \partial_{Q}$ where $Q$ runs over all closed points in the
    normalization of $D$ in $\kappa(D)$, and $\partial_{Q}$ is the
    ramification map with respect to the discrete valuation of
    $\kappa(D)$ defined by $Q$.
  \end{enumerate}
\end{lemma}

\begin{proof}
  The injectivity of $\Br(Z) \to \Br(K(Z))$ in (i) is proven in
  \cite{Milne} IV.2.6, p.145, while the exactness in the middle term
  follows from the purity of the Brauer group (see \cite{AG} 7.4 or
  \cite{Milne} IV.2.18~(b), p.153, and also \cite{Sa08}, Lemma 6.6).
  Item (ii) is proven in \cite{KatoHP}, \S1, p.148.
\end{proof}

In particular, for a $p$-adic curve $X$ (as in
section~\ref{sec:setup}), the residue field $\kappa(x)$ of any closed
point $x\in X_0$ is a $p$-adic field (of characteristic~$0$), thus we
get an exact sequence
\begin{equation}\label{eq:9}
  0 \to \Br(X) \to \Br(K(X)) \stackrel{\bigoplus \partial_x}{\to} \bigoplus_ {x \in X_0} H^1\bigl (\kappa(x), \QQ / \ZZ \bigr)
\end{equation}
which allows us to interpret $\Br(X)$ as the subgroup of $\Br(K(X))$
consisting of unramified classes with respect to all $x\in X_0$.


\subsection{Lichtembaum's duality}

Here we collect some facts about Lichtenbaum's duality (see
\cite{Lichtenbaum} and also \cite{Saito2}, appendix).

\begin{theorem}[Lichtenbaum]\label{thm:lichtembaums-duality}
  Let $X$ be a $p$-adic curve (as in section~\ref{sec:setup}).  Then
  there is a non-degenerate pairing
  $$\langle-, - \rangle \colon  \Pic(X) \times \Br(X) \to \QQ/\ZZ$$
  inducing an isomorphism between $\Br(X)$ and $\Pic(X)^\vee$.
\end{theorem}

This paring can explicitly be described as follows.  For each
$x\in X_0$ write
\begin{itemize}
\item $\inv_{\kappa(x)}\colon \Br(\kappa(x)) \isoto \QQ/\ZZ$ for the
  invariant map with respect to the $p$-adic field $\kappa(x)$;
  
\item $\beta(x)\in \Br(\kappa(x))$ for the image of $\beta\in \Br(X)$
  under the natural map $\Br(X) \to \Br(\kappa(x))$ (induced by the
  inclusion $\Spec \kappa(x) \into X$).
\end{itemize}
Let $\alpha \in \Pic(X)$, $\beta \in \Br(X)$, and let
$\sum_{x\in X_0} n_x x \in \Div(X)$ be a divisor representing
$\alpha$.  Lichtenbaum's pairing is given by
$$\langle \alpha, \beta\rangle
= \sum_{x\in X_0} n_x \inv_{\kappa(x)} \beta(x)
$$
The topology on $\Pic(X)$ is defined as follows.  Let $\Pic^0(X)$ be
the kernel of the degree map $\deg\colon \Pic(X) \to \ZZ$, and $J_X$
be the Jacobian variety of $X$.  There is an exact sequence
(\cite{CornellSilverman}, remark~1.6, p.~169)
$$0 \to \Pic^0(X) \to J_X(K) \stackrel\delta\to \Br(K)\stackrel{\inv_K} = \QQ/\ZZ
$$
where $\delta$ has finite image
$$\im \delta = \frac{\frac1P \ZZ}{\ZZ}
\quad\text{with}\quad P = \text{period of $X$} \df= \gcd \{ \deg
\alpha \mid \alpha \in H^0(K, \Pic(X \tensor_K \overline K)) \}
$$
(in particular, observe that $\Pic^0(X) = J_X(K)$ whenever
$X(K)\ne \emptyset$).  We view $J_X(K)$ as a topological group with
the usual $p$-adic topology, and $\Pic^0(X)$ as a topological subgroup
with the induced topology.  From \cite{Mattuck}, we have

\begin{lemma}[Mattuck]
  Let $A$ be an abelian variety of dimension $g$ over a $p$-adic field
  $K$.  Then $A(K)$ contains an open subgroup of finite index
  isomorphic to $R^g$.
\end{lemma}

Thus the group $J_X(K)$ is compact, and all its finite index subgroups
are open (thus also closed and compact) since any such subgroup
contains a finite index subgroup of $R^g$, which is open.  In
particular, $\Pic^0(X)$ is an open compact subgroup of $J_X(K)$.  We
endow $\Pic(X)$ with the unique topology compatible with its group
structure, and for which $\Pic^0(X)$ is open in $\Pic(X)$.

With this topology, all subgroups $n\Pic(X)$ of $\Pic(X)$
($n =1 ,2, 3, \ldots$) are of finite index and open.  Therefore the
$n$-torsion of $\Br(X)$ is given by
$$\Br(X)[n] \stackrel{\text{Lich}}{=} \Pic(X)^\vee [n] =
\Hom_{c}(\Pic(X), \QQ/\ZZ)[n] = \Hom(\Pic(X)/n, \tfrac{1}{n}\ZZ/\ZZ)
$$

Finally, we mention an important consequence of Lichtembaum's duality,
which is implicit in the proof of theorem~5, p.132 of
\cite{Lichtenbaum}.  See also \cite{Saito2}, \S9 appendix, p.408 and
\cite{Pop}, theorem~4.5, p.172.

\begin{theorem}[Ha\ss e principle]\label{thm:hasse}
  Let $X$ be a $p$-adic curve.  Then the following diagonal map is
  injective:
  $$\Br(K(X)) \into \prod_{x\in X_0} \Br(\widehat{K(X)}_x)$$
  (here $\widehat{K(X)}_x$ is the $v_x$-completion of $K(X)$, see
  notation in section~\ref{sec:setup}).
\end{theorem}


\subsection{Integral models}\label{sec:integral-models}

From now on, we choose once and for all an integral model $\XX$ of
$X$, i.e., a $2$-dimensional regular $R$-scheme such that
\begin{itemize}
\item $f\colon \XX \to \Spec R$ is flat and proper;
\item there is an isomorphism of $\eta$-schemes $X \cong \XX_\eta$;
\item the reduced scheme $(\XX_s)_{\rm red}$ is a $1$-dimensional
  (proper) scheme over $s$ whose irreducible components are all
  regular and have normal crossings (i.e., $\XX_s$ only has ordinary
  double points as singularities).
\end{itemize}
The existence of such an integral model follows from the resolution of
singularities of excellent $2$-dimensional schemes (\cite{Lipman}),
coupled with the embedded resolution of the special fiber (see
\cite{Liu}, corollary~IX.2.30 p.404).  If $X$ admits a \textit{smooth}
integral model $\XX$ over $R$, then we say $X$ has \emph{good
  reduction} over $R$.  If that is the case, then the special fiber
$\XX_s$ will have a single irreducible component, a proper smooth
curve over $s = \Spec k$, and its function field $k(\XX_s)$ will be a
global field of characteritic $p$.

Observe that since $X$ is geometrically integral by assumption,
$H^0(X, \sheaf{O}_X) = K$ by
\cite[\href{https://stacks.math.columbia.edu/tag/0BUG}{Lemma 0BUG}]{stacks-project}.
Then by Stein's factorization
\cite[\href{https://stacks.math.columbia.edu/tag/0AY8}{Lemma 0AY8}]{stacks-project}
the map $f\colon \XX \to \Spec R$ has geometrically connected fibers
and satisfies $f_*\sheaf{O}_{\XX} = \sheaf{O}_R$.

Recall that a \emph{horizontal prime divisor} $D \subset \XX$ is a
prime divisor which is not contained in the special fiber $\XX_s$.
For the next lemma, see \cite{Liu} proposition~VIII.3.4, p.349, and
lemma VIII.3.35, p.360.

\begin{lemma}\label{lemma:integral-models}
  Let $X$ be a $p$-adic curve (as in section~\ref{sec:setup}), and
  $\XX\to \Spec R$ be an integral model.  Then
  \begin{enumerate}[(i)]
  \item A horizontal divisor $D\subset \XX$ is finite and flat over
    $\Spec R$.  In particular, it is affine: $D = \Spec S$ for some
    finite extension of domains $S \supset R$, and the generic point
    $D_\eta$ of $D$ is the spectrum of a $p$-adic field $\Frac S$.
    
  \item There is a bijection between the set of horizontal divisors of
    $\XX$ and $X_0$.  Namely, this bijection maps a horizontal divisor
    $D \subset \XX$ to its generic point $x = D_\eta \in X_0$, and
    conversely the closure $D = \overline{\{x\}}$ of $x\in X_0$ inside
    $\XX$ defines the corresponding horizontal divisor.

  \item For any closed point $P \in \XX_s$, there is a horizontal
    divisor $D$ on $\XX$ intersecting the special fiber $\XX_s$
    transversally at $P$.
  \end{enumerate}
\end{lemma}

The next lemma summarizes the facts we will need concerning the
relation between the Brauer and Picard groups of the special and
generic fibers of $\XX$.

\begin{lemma}\label{lemma:cft}
  In the above notation, let $C_1, \ldots, C_r$ be the irreducible
  components of the special fiber $\XX_s$.
  \begin{enumerate}[(i)]
  \item (Artin) $\Br(\XX) = \Br(\XX_s) = 0$.
  \item Assume that $H^0(\XX_s, \GG_m) = k^\unit$, and let $n\in \NN$
    not divisible by $p$.  Then there are canonical isomorphisms
    \begin{align*}
      \Pic(\XX)[n] &= \Pic(\XX_s)[n]\\
      \Pic(\XX)/n  &= \Pic(\XX_s)/n = H^2(\XX, \mu_n) = H^2(\XX_s, \mu_n)
    \end{align*}

  \item There is an exact sequence
    \begin{center}
      \begin{tikzcd}
        0\arrow{r} & \ZZ \arrow{r}{\iota} & \ZZ^r \arrow{r} &
        \Pic(\XX) \arrow{r}{\res} & \Pic(X) \arrow{r} & 0
      \end{tikzcd}
    \end{center}
    Here $\iota(1) = (a_1, \ldots, a_r)$ is the multiplicity vector of
    the special fiber viewed as an element of $\Div(\XX)$:
    $\XX_s = \divisor(\pi) = \sum_{1\le i\le r} a_i [C_i]$.  In
    particular, if $X$ has good reduction (i.e., $\XX$ is smooth over
    $R$) then restriction $\Pic(\XX) \isoto \Pic(X)$ is an
    isomorphism.
  \end{enumerate}
\end{lemma}

\begin{proof}
  Item (i) follows from the canonical isomorphism
  $\Br(\XX) = \Br(\XX_s)$ (\cite{GB} th\'eor\`eme~3.1, p.~98), and the
  fact that, by class field theory, $\Br(C_i) = 0$ for all $i$, hence
  $\Br(\XX_s) = 0$ as well by \cite{Saltman-p-adic}, lemma~3.2, p.40.

  To show (ii), note that for $i\ge 1$ the Kummer sequence
  $1 \to \mu_n \to \GG_m \stackrel{n}{\to} \GG_m \to 1$ gives a
  commutative diagram with exact rows (the vertical arrows are
  restrictions)
  \begin{center}
    \begin{tikzcd}
      0 \arrow{r}& H^{i-1}(\XX, \GG_m)/n \arrow{d}\arrow{r}&  H^i(\XX,\mu_n) \arrow{d}[swap]{\cong}{\text{PBC}}\arrow{r}&  H^i(\XX,\GG_m)[n] \arrow{d}\arrow{r} & 0\\
      0 \arrow{r}& H^{i-1}(\XX_s, \GG_m)/n \arrow{r}& H^i(\XX_s,\mu_n) \arrow{r}&  H^i(\XX_s,\GG_m)[n] \arrow{r}& 0
    \end{tikzcd}
  \end{center}
  The middle vertical arrow is an isomorphism by proper base change
  (\cite{Milne}, corollary~VI.2.7, p.~224).  For $i=1$, since the
  subgroup $U^{(1)} \df= 1 + (\pi)$ of $R^\unit$ is $n$-divisible by
  Hensel's lemma, from the exact sequence
  $1 \to U^{(1)} \to R^\unit \to k^\unit \to 1$ the leftmost vertical
  arrow is also an isomorphism
  $$H^0(\XX, \GG_m)/n = R^\times/n \isoto k^\unit/n = H^0(\XX_s, \GG_m)/n$$
  (recall that $f\colon \XX \to \Spec R$ satisfies
  $f_*\sheaf{O}_{\XX} = \sheaf{O}_R$ and $H^0(\XX_s, \GG_m) =k^\unit$ by
  hypothesis). Hence the rightmost vertical arrow is an isomorphism
  $\Pic(\XX)[n] \isoto \Pic(\XX_s)[n]$.

  For $i=2$, by (i) both groups $\Br(\XX) = H^2(\XX, \GG_m)$ and
  $\Br(\XX_s) = H^2(\XX_s, \GG_m)$ vanish, and we get the remaining
  isomorphisms to be shown.

  To show (iii), observe that exactness at the two rightmost terms
  follows from \cite{Fulton}, proposition~1.8, p.~21 applied to the
  closed subscheme $\XX_s$ and its open complement
  $X = \XX \smallsetminus \XX_s$, while exactness at the two leftmost
  terms follows from the intersection theory for arithmetic surfaces
  (see \cite{Liu}, \S9.1.2, p.381, in particular theorem~IX.1.23 on
  p.385).
\end{proof}


\subsection{Kummer-Artin-Schreier-Witt theory}\label{sec:kumm-artin-schr}

Let $p$ be a prime number, $\zeta_2\in \mu_{p^2}$ be a primitive
$p^2$-th root of unity, and $A = \ZZ_{(p)}[\zeta_2]$, a discrete
valuation ring with uniformizer $\lambda_2=\zeta_2-1$ and residue
field $\FF_p$.  In \cite{SekiguchiSuwa}, Sekiguchi and Suwa
constructed an exact sequence of group schemes over $A$
\begin{equation}\label{eq:3}
0 \longrightarrow \ZZ/ p^{2} \longrightarrow \mathcal{W}_{2} \stackrel{\Psi}{\longrightarrow} \mathcal{V}_{2} \rightarrow 0
\end{equation}
whose special fiber is isomorphic to the Artin-Schreier-Witt sequence
in characteristic $p>0$, while its generic fiber is isomorphic to a
Kummer type sequence, thus giving an explicit interpolation between
the Artin-Schreier-Witt and Kummer theories.  In order to describe
Sekiguchi and Suwa's construction, set
$$ 
\zeta = \zeta_2^p,\quad \lambda=\zeta-1,\quad \lambda_2=\zeta_2-1,\quad
\eta=\sum_{k=1}^{p-1} \frac{(-1)^{k-1}}{k} \lambda_{2}^{k}
,\quad\tilde{\eta}=\frac{\lambda^{p-1}}{p}(p \eta-\lambda)
$$
and consider the polynomials in $A[T]$ ($p$-truncated exponential series)
\begin{equation}\label{eq:11}
F(T)=\sum_{k=0}^{p-1} \frac{(\eta T)^{k}}{k !}, \quad G(T)=\sum_{k=0}^{p-1} \frac{(\tilde\eta T)^{k}}{k !}
\end{equation}
The group schemes $\mathcal{W}_{2}$ and $\mathcal{V}_{2}$ are given by
\begin{align*}
  \mathcal{W}_{2} & =\Spec A\left[T_{0}, T_{1}, \frac{1}{\lambda T_{0}+1}, \frac{1}{\lambda T_{1}+F\left(T_{0}\right)}\right] \\
  \mathcal{V}_{2} &=\Spec A\left[T_{0}, T_{1}, \frac{1}{\lambda^{p} T_{0}+1}, \frac{1}{\lambda^{p} T_{1}+G\left(T_{0}\right)}\right]
\end{align*}
The comultiplication in $\mathcal{W}_{2}$ is given by 
$$ 
\left(T_{0}, T_{1}\right) \mapsto\left(\Lambda_{0}^{F}\left(T_{0}
    \otimes 1,1 \otimes T_{0}\right), \Lambda_{1}^{F}\left(T_{0}
    \otimes 1, T_{1} \otimes 1,1 \otimes T_{0}, 1 \otimes
    T_{1}\right)\right)
$$
and the one in $\mathcal{V}_{2}$ by
$$ 
\left(T_{0}, T_{1}\right) \mapsto\left(\Lambda_{0}^{G}\left(T_{0}
    \otimes 1,1 \otimes T_{0}\right), \Lambda_{1}^{G}\left(T_{0}
    \otimes 1, T_{1} \otimes 1,1 \otimes T_{0}, 1 \otimes
    T_{1}\right)\right)
$$
where $\Lambda_i^F$ and $\Lambda_i^G$ are the polynomials with
coefficients in $A$ given by
$$
\Lambda_{0}^{F}\left(X_{0}, Y_{0}\right)=\lambda X_{0} Y_{0}+X_{0}+Y_{0}, \quad \Lambda_{0}^{G}\left(X_{0}, Y_{0}\right)=\lambda^{p} X_{0} Y_{0}+X_{0}+Y_{0}
$$
\begin{align*}
  \Lambda_{1}^{F}\left(X_{0}, X_{1}, Y_{0}, Y_{1}\right) &= \lambda X_{1} Y_{1}+X_{1} F\left(Y_{0}\right)+F\left(X_{0}\right) Y_{0}\\
                                                         &\qquad+\frac{1}{\lambda}\left[F\left(X_{0}\right) F\left(Y_{0}\right)-F\left(\lambda X_{0} Y_{0}+X_{0}+Y_{0}\right)\right]\\
\Lambda_{1}^{G}\left(X_{0}, X_{1}, Y_{0}, Y_{1}\right)&= \lambda^{p} X_{1} Y_{1}+X_{1} G\left(Y_{0}\right)+G\left(X_{0}\right) Y_{0} \\
                                                         &\qquad+\frac{1}{\lambda^{p}}\left[G\left(X_{0}\right) G\left(Y_{0}\right)-G\left(\lambda^{p} X_{0} Y_{0}+X_{0}+Y_{0}\right)\right]
\end{align*}
Finally the isogeny $\Psi\colon \mathcal{W}_{2} \to \mathcal{V}_{2}$
is given by $(T_{0}, T_{1}) \mapsto\left(\Psi_{0}\left(T_{0}\right), \Psi_{1}\left(T_{0}, T_{1}\right)\right)$
where $\Psi_i \in A\left[T_{0}, T_{1}, 1 /\left(\lambda  T_{0}+1\right), 1 /\left(\lambda  T_{1}+F\left(T_{0}\right)\right)\right]$ 
are given by
\begin{align*}
  \Psi_{0}\left(T_{0}\right) &=\frac{\left(\lambda T_{0}+1\right)^{p}-1}{\lambda^{p}}\\
\Psi_{1}\left(T_{0}, T_{1}\right) &=\frac{1}{\lambda^{p}}\left[\frac{\left(\lambda T_{1}+F\left(T_{0}\right)\right)^{p}}{\lambda T_{0}+1}-G\left(\frac{\left(\lambda T_{0}+1\right)^{p}-1}{\lambda^{p}}\right)\right]
\end{align*}
As shown in \cite{SekiguchiSuwa}, the special fiber of \eqref{eq:3} is
isomorphic to the Artin-Schreier-Witt sequence (here $W_2$ denotes the
group of Witt vectors of length~$2$ and $F$, the Frobenius map)
$$ 
0 \longrightarrow \ZZ/p^{2} \longrightarrow W_{2} \stackrel{F-\id}{\longrightarrow} W_{2} \longrightarrow 0
$$
Moreover, there is a commutative diagram of group schemes over $A$
\begin{center}
  \begin{tikzcd}
    \mathcal{W}_{2} \arrow{d}[swap]{\alpha^{(F)}} \arrow{r}{\Psi}& \mathcal{V}_{2} \arrow{d}{\alpha^{(G)}}\\
    \GG_{m, A}^2 \arrow{r}{\Theta} & \GG_{m, A}^2
  \end{tikzcd}
\end{center}
whose vertical arrows become isomorphisms over $\Spec (\Frac A)$.
Here, if $\GG_{m, A}^2 = \Spec A[U_0, U_1, 1/U_0, 1/U_1]$, the maps
$\alpha^{(F)}$ and $\alpha^{(G)}$ are respectively given by
$$\left(U_{0}, U_{1}\right) \mapsto\left(\lambda T_{0}+1, \lambda
  T_{1}+F\left(T_{0}\right)\right)
\quad\text{and}\quad
\left(U_{0}, U_{1}\right) \mapsto\left(\lambda^{p} T_{0}+1,
  \lambda^{p} T_{1}+G\left(T_{0}\right)\right)
$$
and $\Theta$, by
$(U_{0}, U_{1}) \mapsto (U_{0}^{p}, U_{0}^{-1} U_{1}^{p})$.

Let us give the local description of how to lift Artin-Schreier-Witt
extensions to Kummer ones.  Consider the polynomial
\begin{equation}\label{eq:10}
c(X, Y)=\frac{X^{p}+Y^{p}-(X+Y)^{p}}{p} \in \ZZ[X, Y]
\end{equation}
Let $B$ be an $A$-algebra, and write
$\overline B \df= B\tensor_A \FF_p = B/(\lambda_2)$ for its reduction
modulo $\lambda_2$.  To give an Artin-Schreier-Witt extension of
$\overline B$ is the same as to give a morphism
$\overline f\colon \Spec \overline B \to \mathcal{V}_2\tensor_A \FF_p
= \Spec \FF_p[T_0, T_1]$, i.e., to give a pair
$(\overline b_0, \overline b_1)\in \overline B^2$ so that
$\overline f$ is defined by
$(T_0, T_1)\mapsto (\overline b_0, \overline b_1)$; then the
corresponding Artin-Schreier-Witt extension is given by taking the
pullback
$\Spec(\overline B)\times_{\mathcal{V}_2} \mathcal{W}_2 \to \Spec
(\overline B)$ of $\Psi$, i.e., it is the extension
$\overline B[x_0, x_1] \supset \overline B$ defined by equations
$$x_0^p - x_0 = \overline b_0,\qquad x_1^p - x_1 - c(x_0^p, - x_0) =
\overline b_1
$$
Any lift $f\colon \Spec B \to \mathcal{V}_2$ of $\overline f$ now
defines a Kummer extension lifting the above Artin-Schreier-Witt one.
Namely, let $(b_0, b_1) \in B^2$ be any lift
$(\overline b_0, \overline b_1)$ with $\lambda^{p} b_{0}+1$ and
$\lambda^{p} b_{1}+G(b_{0})$ units in $B$.  Then
$(T_0, T_1)\mapsto (b_0, b_1)$ defines a morphism
$f\colon \Spec B \to \mathcal{V}_2$, and taking the pullback
$\Spec(B)\times_{\mathcal{V}_2} \mathcal{W}_2 \to \Spec(B)$ of $\Psi$
with respect to $f$ we get a Kummer extension $B[t_0, t_1] \supset B$
via equations
$$\frac{\left(\lambda t_{0}+1\right)^{p}-1}{\lambda^{p}} = b_0,\qquad
\frac{1}{\lambda^{p}}\left[\frac{\left(\lambda t_{1}+F\left(t_{0}\right)\right)^{p}}{\lambda t_{0}+1}-G(b_0)\right]=b_1
$$
or, written in a more Kummerish style,
$$(\lambda t_{0}+1)^{p} = 1 + b_0\lambda^{p},\qquad
(\lambda t_{1}+F\left(t_{0}\right))^{p} = (\lambda t_{0}+1)(b_1\lambda^{p} +G(b_0))
$$
so that $B[t_0, t_1] \supset B$ is built by succesively taking $p$-th
roots of certain elements.


\section{Cyclicity for curves with good reduction}

In this section, our goal is to show

\begin{theorem}\label{thm:main}
  Assume that $X$ has good reduction, and let $\XX \to \Spec R$ be a
  smooth integral model.  Let $n\in \NN$ with $p\nmid n$.  The
  ramification map with respect to the prime divisor $\XX_s$ induces
  an isomorphism
  $$\Br(X)[n] = H^1(\XX_s, \ZZ/n)$$
  Hence all elements of $\Br(X)[n]$ are cyclic, namely given any
  uniformizer $\pi\in K$, they have the form
  $\delta_n\pi \cup \tilde\chi$ where
  $\tilde \chi \in H^1(\XX, \ZZ/n)$ is the unique lift of the
  ramification $\chi \in H^1(\XX_s, \ZZ/n)$ of $\beta$ with respect to
  $\XX_s$.
\end{theorem}

\begin{proof}
  Let $k(\XX_s)$ be the function field of the special fiber $\XX_s$,
  and denote by
  $\partial_s\colon \Br(K(X))[n] \to H^1(k(\XX_s),\ZZ/n)$ the
  corresponding ramification map.  First we show that $\partial_s$
  restricts to a map
  $\partial_s\colon \Br(X)[n] \to H^1(\XX_s, \ZZ/n)$, i.e., that for
  $\beta \in \Br(X)[n]$ the character $\chi = \partial_s (\beta)$ is
  unramified with respect to any closed point $P\in\XX_s$.  From
  sequence~\eqref{eq:9} we get $\partial_x (\beta) = 0$ for all
  $x\in X_0$, thus also $\partial_D (\beta) = 0$ for all horizontal
  prime divisors $D$ of $\XX$ by
  lemma~\ref{lemma:integral-models}(ii).  Applying
  lemma~\ref{lemma:Brauer}(ii) with $A = \sheaf{O}_{\XX,P}$ we get
  that $\partial_P (\chi) = 0$, i.e., $\chi$ is unramified at $P$, as
  desired.

  Next, given $\chi \in H^1(\XX_s, \ZZ/n)$ let
  $\tilde\chi\in H^1(\XX, \ZZ/n)$ be its unique lift under the proper
  base change isomorphism $H^1(\XX, \ZZ/n) = H^1(\XX_s, \ZZ/n)$
  (\cite{Milne}, corollary~VI.2.7, p.~224).  Since $\pi$ is a
  uniformizer for the valuation defined by $\XX_s$ and $\tilde \chi$
  is $\pi$-unramified, we have
  $\partial_s(\delta_n \pi \cup \tilde\chi) = \chi$, showing that
  $\partial_s\colon \Br(X)[n] \onto H^1(\XX_s, \ZZ/n)$ is surjective.
  Finally, if $\beta \in \Br(X)[n]$ is such that
  $\partial_s(\beta) = 0$, then $\beta$ in unramified on the whole
  scheme $\XX$, thus $\beta \in \Br(\XX) = 0$ by
  lemmas~\ref{lemma:Brauer}(i) and \ref{lemma:cft}(i).  This proves
  $\partial_s\colon \Br(X)[n] \into H^1(\XX_s, \ZZ/n)$ is also
  injective.
\end{proof}

\begin{remark}
  The previous theorem can also be obtained using Lichtembaum's
  duality.  We briefly sketch the proof.  First recall that from Class
  Field Theory of global fields of positive characteristic (see
  \cite{MilneAD}, chapter~I, appendix~A), for any smooth projective
  curve $Z$ over a finite field $k$, there is a non-degenerate pairing
  of finite groups
  \begin{equation}\label{eq:15}
    H^1(Z, \ZZ/n) \times \Pic(Z)/n \to \ZZ/n
  \end{equation}
  yielding an isomorphism $H^1(Z, \ZZ/n) = (\Pic(Z)/n)^\vee$.  This
  pairing is induced by the local perfect pairings
  $$H^1(\widehat{k(Z)}_z, \ZZ/n) \times H^1(\widehat{k(Z)}_z, \mu_n)\to
  \Br(\widehat{k(Z)}_z)[n] = \ZZ/n
  $$
  (here $\widehat{k(Z)}_z$ stands for the completion of the function
  field $k(Z)$ of $Z$ with respect to the valuation given by a closed
  point $z\in Z_0$).
  
  Explicitly, given $\chi \in H^1(Z, \ZZ/n)$ and a divisor
  $\sum_z n_z z \in \Div(Z)$ representing a given
  $\alpha\in \Pic(Z)/n$, the pairing \eqref{eq:15} maps
  $(\chi, \alpha)$ to
  $\sum_z n_z \cdot \chi|_{\kappa(z)}(\mathord{\text{Frob}}_z)$ where
  $\mathord{\text{Frob}}_z\in G_{\kappa(z)}$ is the Frobenius
  automorphism (here $\kappa(z)$ denotes the residue field of $z$, a
  finite field).

  Now to get the theorem we just splice together Lichtenbaum's duality
  (theorem~\ref{thm:lichtembaums-duality}), lemma~\ref{lemma:cft}, the
  isomorphism $H^1(\XX_s, \ZZ/n) = (\Pic(\XX_s)/n)^\vee$ given by
  \eqref{eq:15} and the proper base change theorem:
  \begin{align*}
    \Br(X)[n] &\stackrel{\ref{thm:lichtembaums-duality}}{=} \Hom(\Pic(X)/n, \ZZ/n) \stackrel{\ref{lemma:cft}(iii)}{=} \Hom(\Pic(\XX)/n, \ZZ/n)\\
              &\stackrel{\ref{lemma:cft}(ii)}{=} \Hom(\Pic(\XX_s)/n,
                \ZZ/n) \stackrel{\eqref{eq:15}}{=} H^1(\XX_s, \ZZ/n) \stackrel{\text{\rm PBC}}{=} H^1(\XX, \ZZ/n)
  \end{align*}
  Given the explict descriptions of \eqref{eq:15} and Lichtembaum's
  pairing, it is not difficult to check that this chain of
  isomorphisms is just the ramification map $\partial_s$ above.
\end{remark}


\section{Non-cyclic division algebras in the singular case}\label{sec:non-cyclic-division}

In this section, we show that for a general $p$-adic curve $X$ with
bad reduction, one cannot expect $\Br(X)[n]$ to be comprised solely of
$\ZZ/n$-cyclic algebras.  In example~\ref{eg:noncyclic} below, for
infinitely many primes $p$, we exhibit a non $\ZZ/3$-cyclic class of
period~$3$ in the Brauer group of a certain $p$-adic curve of
genus~$10$.  Both for conceptual clarity and in order to facilitate
the construction of other similar examples, we work in a more general
situation, as described in the following

\begin{setup}\label{setup:ceg}
  Let $d\in \NN$ be prime to $p$, and suppose that $\mu_d \subset K$.
  Choose homogenous polynomials $F_1, F_2\in R[x, y, z]$ with
  $\deg F_i = d$ such that
  \begin{enumerate}[(i)]
  \item in $\PP^2_R$ the closed subschemes $V_+(F_1)$ and $V_+(F_2)$
    intersect the line $V_+(z)$ at two $R$-points $\infty_1$,
    $\infty_2$, each with multiplicity $d$, and such that
    $\overline \infty_1 \ne \overline \infty_2$.  Here $\overline
    \infty_i$ denotes the special fiber of $\infty_i$;

  \item the reductions $\overline F_1, \overline F_2 \in k[x, y, z]$
    modulo $\pi$ define geometrically connected smooth curves over $k$
    $$C_1 \df= \Proj k[x, y, z]/(\overline F_1)
    \qquad
    C_2\df= \Proj k[x, y, z]/(\overline F_2)
    $$
    having normal crossings in $\PP^2_{k}$, with at least one
    $k$-rational common point $P_0\in C_1\cap C_2$;

  \item $\Pic(C_1)[d] \ne 0$.
  \end{enumerate}
  Define the integral projective $R$-scheme
  $$\XX \df= \Proj \frac{R[x, y, z]}{(F_1\cdot F_2 - \pi z^{2d})} \to \Spec R
  $$
  and the elements $f_1 \df= F_1/z^d , f_2 \df= F_2/z^d \in K(\XX)$ in
  its function field $K(\XX)$.  Finally, let $u\in R^\unit$ be any
  unit which is not a $d$-th power.
\end{setup}

\begin{example}\label{eg:noncyclic}
  As a concrete example, we may choose $d=3$ and a prime
  $p\equiv 1 \pmod 3$, so that there is a primitive third root of
  unity $\omega\in \ZZ_p$ by Hensel's lemma.  We will show that for
  almost all $p$ the polynomials in $\ZZ_p[x, y, z]$
  \begin{align*}
    F_1 &= y^2z + yz^2 - x^3 - x^2z + 7xz^2 -5 z^3\\
    F_2 &= x^2z - y^3 + 26z^3
  \end{align*}
  will satisfy the conditions of the setup (the curve
  $V_+(F_1) \subset \PP^2_{\QQ}$ is the elliptic curve labeled 91.b2
  in the LMFDB database, see \texttt{http://www.lmfdb.org}).  Let
  $\XX = V_+(F_1F_2 - pz^6) \subset \PP_{\ZZ_p}^2$ and
  $f_i = F_i/z^3 \in K(\XX)$ as above.  Then, in $\PP^2_{\ZZ_p}$, the
  subschemes $V_+(F_1)$ and $V_+(F_2)$ intersect the line $V_+(z)$ at
  the $R$-points $\infty_1 = V_+(x,z)$ and $\infty_2 = V_+(y,z)$, both
  with multiplicity~$3$.

  If $p\ne 2,3,7,13$ then both reductions $\overline F_i$ will be
  elliptic curves over $\FF_p$ sharing a common $\FF_p$-rational point
  $P_0 = (-1:3:1) \in C_1\cap C_2$.  Moreover the resultant of
  $F_1(x, y, 1)$ and $F_2(x, y, 1)$ with respect to the variable $x$
  is
  \begin{equation}\label{eq:13}
    (y - 3)(y^8 + 3y^7 + 9y^6 - 66y^5 - 196y^4 - 587y^3 + 1084y^2 +
    3209y + 9585)
  \end{equation}  
  which is separable over $\FF_p$ as long as
  $p\ne 2, 3, 37, 97, 29723, 447692787897013$.  Then by Bézout's
  theorem each of the $3^2 = 9$ roots of \eqref{eq:13} will correspond
  to a different multiplicity~$1$ intersection point of $C_1$ and
  $C_2$, thus for almost all $p$ the curves $C_1$ and $C_2$ will be
  regular and have normal crossings.  Finally, $\Pic(C_1)[3] \ne 0$;
  in fact, if $P_1 = (1:0:1)$ then
  $[P_1] - [\overline\infty_1] \in \Pic^0(C_1)[3]$ is a nontrivial
  element.
\end{example}

We will need the following result concerning the Picard group of
reducible curves such as $\XX_s$ (see \cite{Hida}, theorem~4.1.5,
chapter~4, p.296):

\begin{lemma}\label{lemma:non-cyclic-division}
   Let $k$ be a field, and $C = C_1 \cup C_2$ be the union of two
   proper smooth irreducible curves over $k$ such that its components
   intersect transversally at $r$ points over a finite extension of
   fields $l \supset k$.  Then there is a split exact sequence
  $$0 \to T(k) \to \Pic^0(C) \to \Pic^0(C_1) \oplus \Pic^0(C_2) \to 0
  $$
  where $T$ is an algebraic torus that becomes isomorphic to
  $\GG_m^{r-1}$ over $l$.
\end{lemma}

\begin{lemma}
  Assume setup~\ref{setup:ceg}.  Then
  \begin{enumerate}[(a)]
  \item $\XX$ is a regular scheme;

  \item The generic fiber
    $X \df= \XX_\eta = \Proj K[x, y, z]/(F_1F_2 - \pi z^{2d})$ of
    $\XX$ is a smooth projective curve over $K$;

  \item The special fiber
    $\XX_s = \Proj k[x, y, z]/(\overline F_1 \overline F_2)$ of $\XX$
    is a reduced $k$-scheme whose irreducible components are $C_1$ and
    $C_2$;
  
  \item The divisors $[C_1]$ and $[C_2]$ in $\Div(\XX)$ are principal
    divisors modulo $d$, more precisely
    \begin{align*}
      \divisor(f_1) &= [C_1] + d^2\cdot [\infty_1] - d^2\cdot [\infty_2]\\
      \divisor(f_2) &= [C_2] + d^2\cdot [\infty_2] - d^2\cdot [\infty_1]
    \end{align*}
    In particular, $\Pic(\XX)$ can be generated by horizontal divisors
    only (see definition before lemma~\ref{lemma:integral-models});
  
  \item There exists $h\in K(X)^\unit$ such that
    \begin{itemize}
    \item $\divisor(h) = d E$ for some $E\in \Div(\XX)$;
    \item the restriction $\overline h\in k(C_1)^\unit$ to the function
      field of $C_1$ is well-defined;
    \item $\overline h$ is not a $d$-th power in $k(C_1)^\unit$;
    \item $P_0$ is not a pole of $\overline h$, and
      $\overline h(P_0) = \overline 1 \in k(P_0)^\unit$ (here $k(P_0)$
      denotes the residue field of $P_0$).
    \end{itemize}
  \end{enumerate}
\end{lemma}

\begin{proof}
  To prove (a) we have to show that for all closed points $P\in \XX$
  the maximal ideal $\ideal{m}_P$ of $\sheaf{O}_{\XX, P}$ (a
  $2$-dimensional local ring) can be generated by $2$ elements.  Since
  $P$ necessarily lies on the special fiber $\XX_s$ there are two
  cases to consider:
  \begin{itemize}
  \item $P$ is a regular point of the special fiber $\XX_s$.  In that
    case, if $\overline t$ is a uniformizer of the discrete valuation
    ring $\sheaf{O}_{\XX_s, P} = \sheaf{O}_{\XX, P}/(\pi)$ then
    $\ideal{m}_P = (t, \pi)$ for any lift $t\in \sheaf{O}_{\XX, P}$ of
    $\overline t$;

  \item $P$ is one of the nodal points in $C_1\cap C_2$.  Then
    $\ideal{m}_P = (f_1, f_2, \pi)$, which is clearly equal to
    $(f_1, f_2)$ since $f_1f_2 = \pi$ in $\sheaf{O}_{\XX, P}$.
  \end{itemize}
  Since $\XX$ is regular, so is its open subset $X$, therefore (b)
  holds.  Item (c) is clear.

  To show (d), we work with $f_1$, the case for $f_2$ being similar.
  First notice that, as an element of the ring
  $\sheaf{O}_{\XX}(D_+(z))$, $f_1$ is a prime defining the generic
  point of $C_1$ since (here $\overline f_1$ stands for
  $f_1\bmod \pi$)
  $$\frac{\sheaf{O}_{\XX}(D_+(z))}{(f_1)}
  = \frac{R[\frac{x}{z}, \frac{y}{z}]}{(f_1f_2 - \pi, f_1)}
  = \frac{k[\frac{x}{z}, \frac{y}{z}]}{(\overline f_1)}
  = \sheaf{O}_{C_1}(D_+(z))
  $$
  The complement $V_+(z) = \XX \smallsetminus D_+(z)$ of the open
  subset $D_+(z)$ consists of two $R$-points $\infty_1$ and
  $\infty_2$, hence the support of $\divisor(f_1)$ is contained in
  $\{ C_1, \infty_1, \infty_2 \}$, and all that is left is to compute
  the valuations $v_{\infty_i}(f_1)$ of $f_1$ with respect to the
  points $\infty_i$.  Actually, it suffices to show that
  $v_{\infty_1}(f_1) = d^2$; then by symmetry
  $v_{\infty_2}(f_2) = d^2$ and from $f_1f_2 = \pi$ we will
  automatically get $v_{\infty_2}(f_1) = -d^2$.

  To show $v_{\infty_1}(f_1) = d^2$, it is enough to work over the
  $2$-dimensional local ring $\sheaf{O}_{\XX, \overline\infty_1}$.
  Without loss of generality, say that $\infty_1 \in D_+(y)$, and
  write $x' = x/y$, $z' = z/y$ and $\phi_i = F_i/y^d \in R[x', z']$.
  Let $\xi\in \Spec R[x', z'] \subset \PP^2_R$ be the prime
  corresponding to the generic point of $\infty_1$ (so that
  $\overline\infty_1 = (\pi, \xi)$ and
  $\sheaf{O}_{\XX, \overline\infty_1} = R[x', z']_{(\pi, \xi)}/(\phi_1
  \phi_2 - \pi (z')^{2d})$).  Since
  $\overline\infty_1\ne \overline\infty_2$ by hypothesis,
  $\overline\infty_1 \notin V_+(\overline F_2) \subset \PP^2_k$ and
  thus $\phi_2$ is a unit in $\sheaf{O}_{\XX, \overline\infty_1}$.
  Therefore
  $$\phi_1 \phi_2 = \pi (z')^{2d} \implies
  f_1 = \frac{\phi_1}{(z')^d} = \frac{\pi}{\phi_2}\cdot  (z')^{d}
  \qquad\text{ in }
  \sheaf{O}_{\XX, \overline\infty_1}
  $$
  showing that $v_{\infty_1}(f_1) = d\cdot v_{\infty_1}(z')$.
  Finally, to show that $v_{\infty_1}(z') = d$, we have to compute the
  length of the $\sheaf{O}_{\XX, \xi}$-module
  $$\frac{\sheaf{O}_{\XX, \xi}}{(z')}
  = \frac{R[x', z']_{\xi}}{(\phi_1 \phi_2 - \pi (z')^{2d}, z')} =
  \frac{R[x', z']_{\xi}}{(\phi_1, z')}
  $$
  (since $\phi_2\in \sheaf{O}_{\XX, \overline\infty_1}^\unit$,
  $\phi_2 \notin (\pi, \xi)$ in $R[x', z']$ and thus
  $\phi_2\in R[x', z']_\xi^\unit$ as well).  But the length of the
  last module is exactly the intersection multiplicity of $V_+(z)$ and
  $V_+(F_1)$ in $\PP^2_R$ at $\infty_1$, which is $d$ by hypothesis.
  
  Finally, to show (e) let $\overline L_1 \in \Pic(C_1)[d]$ be a
  nontrivial element, whose existence is guaranteed by
  setup~\ref{setup:ceg}(iii).  Since $\XX_s$ is geometrically
  connected, $H^0(\XX_s, \sheaf{O}_{\XX_s}) = k$ and by
  lemmas~\ref{lemma:cft}(ii) and \ref{lemma:non-cyclic-division}, the
  natural map
  $$\Pic(\XX)[d] = \Pic(\XX_s)[d] \onto \Pic^0(C_1)[d] \oplus \Pic^0(C_2)[d]$$
  is surjective, hence there is $L\in \Pic(\XX)[d]$ restricting to
  $\overline L_1$.  We may write $L = \sheaf{O}_{\XX}(E)$ for some
  divisor $E\in \Div(\XX)$ whose support does not contain $C_1$, $C_2$
  or any divisor going through any of the (at most $d^2$) intersection
  points of $C_1$ and $C_2$.  In fact, $\Pic(\XX)$ is generated by
  horizontal divisors; moreover, as in the proof of the lemma in
  \cite{Saltman-p-adic-c}, we may modify $E$ by a principal divisor
  using the fact that there is an affine open set $\Spec A$ of $\XX$
  containing all these intersection points, and that the
  semi-localization of $A$ with respect to the corresponding maximal
  ideals is a UFD (by Auslander-Buchsbaum's theorem and
  \cite{Matsumura}, exercise 20.5, p.169).

  Since $L = \sheaf{O}_{\XX}(E)$ has order $d$ in $\Pic(\XX)$, there
  exists $h\in K(X)^\unit$ such that $\divisor(h) = dE$, and by the
  choice of $E$ the restriction $\overline h \in k(C_1)^\unit$ is
  well-defined, and $P_0$ is neither a zero nor a pole of
  $\overline h$.  Thus multiplying $h$ by a unit in $R^\unit$, we may
  further assume that $\overline h(P_0) = \overline 1$.  Now all that
  is left is to show that $\overline h$ is not a $d$-th power in
  $k(C_1)^\unit$.  Write $\overline E_1\in \Div(C_1)$ for the
  restriction of $E$ to $C_1$.  Then
  $\overline L_1 = \sheaf{O}_{C_1}(\overline E_1)$ in $\Pic(C_1)$ and
  $\divisor(\overline h) = d \overline E_1$.  Hence if
  $\overline h = \overline g^d$ with $\overline g\in k(C_1)^\unit$ we
  would have $\divisor (\overline g) = \overline E_1$, which
  contradicts the fact that
  $\overline L_1 = \sheaf{O}_{C_1}(\overline E_1)$ is not trivial in
  $\Pic(C_1)$.  This finishes the proof of item (e).
\end{proof}

Now we are ready to show

\begin{claim}
  Assume setup~\ref{setup:ceg}, and let $h$ be as in the previous
  lemma.  Choose a primitive root of unity $\zeta_d \in \mu_d$, and
  define the period $d$ sum of symbol algebras
  $$\beta \df= (f_1, h)_{\zeta_d} + (f_2, u)_{\zeta_d} \in \Br(K(X))[d]$$
  Then $\beta \in \Br(X)[d]$.  If $m$ and $n$ are respectively the
  orders of $\overline h$ in $k(C_1)^\unit/d$ and $u$ in $k^\unit/d$
  then $mn \mid \ind(\beta)$.  In particular, when $d$ is prime
  $\ind(\beta) > d$, hence $\beta$ is not $\ZZ/d$-cyclic.
\end{claim}

\begin{proof}
  Let $P\in X_0$ be an arbitrary closed point of the generic fiber of
  $\XX$.  Then $P$ is the generic point of a unique horizontal divisor
  on $\XX$ (lemma~\ref{lemma:integral-models}(ii)); write
  $v_P \colon K(X) \to \ZZ \cup \{\infty\}$ and
  $\partial_P \colon \Br(K(X))[d] \to H^1(\kappa(P), \QQ/\ZZ)[d]$ for
  the corresponding discrete valuation and ramification maps.  By
  sequence \eqref{eq:9}, to prove that $\beta \in \Br(X)[d]$, we have
  to show that $\partial_P(\beta) = 0$, but this is clear since all
  $v_P(f_1)$, $v_P(f_2)$, $v_P(h)$, $v_P(u) = 0$ are multiples of $d$
  by itens (d) and (e) of the previous lemma.

  Next let $\widehat{K(X)}$ be the completion of $K(\XX) = K(X)$ with
  respect to the valuation defined by the prime divisor $C_1$.
  Observe that $f_1$ is a uniformizer of $\widehat{K(X)}$, and that
  its residue field is $k(C_1)$, the function field of $C_1$.  To show
  that $mn \mid \ind(\beta)$ it is enough to show that the index of
  the restriction $\beta|_{\widehat{K(X)}}$ is a multiple of $mn$.  We
  just have to apply Nakayma's index formula (lemma~\ref{lemma:NAK}).
  Denoting by a bar images in $k(C_1)$, since the order of the
  $f_1$-unramified Kummer character given by
  $\delta_d(\overline h) \in H^1(k(C_1), \mu_d) \subset
  H^1(\widehat{K(X)}, \mu_d)$ is $m$,
  $$\ind \beta|_{\widehat{K(X)}}
  = m \cdot \ind (\overline f_2, \overline u)_{\zeta_d}|_{k(C_1)(\overline h^{1/d})}
  $$
    Since
  $k(C_1)(\overline h^{1/d})$ is a global field of characteristic $p$,
  by Class Field Theory the index of
  $(\overline f_2, \overline u)_{\zeta_d}|_{k(C_1)(\overline
    h^{1/d})}$ is the least common multiple of its local invariants,
  hence it suffices to show that one of these invariants equals $n$.
  Since $\overline h(P_0) = \overline 1$ by construction, $P_0$ splits
  completely in the Kummer extension $k(C_1)(\overline h^{1/d})$,
  hence the local invariant of
  $(\overline f_2, \overline u)_{\zeta_d}|_{k(C_1)(\overline
    h^{1/d})}$ with respect to any point lying over $P_0$ is the same
  as $\inv_{P_0}(\overline f_2, \overline u)_{\zeta_d}$.  But the
  latter is equal to the order $n$ of $\overline u$ in $k^\unit/d$
  since $\overline f_2$ is a uniformizer of $P_0\in C_1$ (recall that
  $C_1$ and $C_2$ have normal crossings by assumption), and we are
  done.
\end{proof}


\section{Tate curves}\label{sec:tate-curves}

In this section, we use Lichtenbaum's duality in order to study the
Brauer group of elliptic curves with split multiplicative reduction,
i.e., Tate curves.  The nice fact about Tate curves is that they enjoy
a $p$-adic uniformization, which will allow us to obtain cyclicity
results even for the $p$-primary case.

Let $q\in K^\unit$ with $|q| < 1$. Recall that the \emph{Tate elliptic
  curve} $E_q$ is the elliptic curve defined by the affine equation
$$y^2 + xy = x^3 + a_4(q) x + a_6(q)$$
where
$$s_r(q) = \sum_{n\ge 1} \frac{n^r q^n}{1-q^n}\qquad
a_4(q) = - 5s_3(q)\qquad
a_6(q) = - \frac{5s_3(q) +7s_5(q)}{12}
$$
For the next result, see \cite{SilvermanATAEC}, theorems~V.3.1 on
p.~423, V.5.3 on p.442.

\begin{theorem}[Tate]\label{thm:tateellcurve}
  Let $q\in K^\unit$ with $|q| < 1$, and $E_q$ be the Tate elliptic
  curve as defined above.
  \begin{enumerate}[(i)]
  \item There is an exact sequence of $G_K$-modules
    \begin{center}
      \begin{tikzcd}[column sep=large]
        0\arrow{r}& \ZZ \arrow{r}{n\mapsto q^n}& \overline K^\unit \arrow{r}{\varphi}& E_q(\overline K) \arrow{r}& 0
      \end{tikzcd}
    \end{center}
    with the usual trivial $G_K$-action on $\ZZ$, and Galois actions
    on $\overline K^\unit$ and $E_q(\overline K)$.  In particular, for
    every algebraic extension $L \supseteq K$, $\varphi$ induces an
    isomorphism $\varphi\colon E_q(L) \isoto L^\unit/q^\ZZ$ (since
    $H^1(G_L, \ZZ) = \Hom_c(G_L, \ZZ) = 0$).

  \item Let $E$ be an elliptic curve over $K$ with split
    multiplicative reduction.  Then $E$ is $K$-isomorphic to $E_q$ for
    some $q\in K^\unit$, $|q|<1$.
  \end{enumerate}    
\end{theorem}

The main technical result in this section is given by the following

\begin{lemma}\label{lemma:brgptatecurve}
  Let $q\in K^\unit$ with $|q| < 1$, and $E_q$ be the corresponding
  Tate elliptic curve.  For any finite field extension $L\supseteq K$,
  let $\theta_L \colon L^\unit \into G_L^{\rm ab}$ be the local
  reciprocity map (see \cite{Neukirch}, chapter V, p.317 and
  \cite{SerreLF}, chapter XIII, p.188) and write $H^1_q(G_L, \QQ/\ZZ)$
  for the kernel of
  \begin{align*}
    \epsilon_q\colon H^1(G_L, \QQ/\ZZ) =  \Hom_c(G_L^{\rm ab}, \QQ/\ZZ)&\to  \QQ/\ZZ\\
    \chi&\mapsto \chi(\theta_L(q))
  \end{align*}
  Then there is an isomorphism
  $$\xi_q\colon \Br(E_q \tensor L)\isoto H^1_q(G_L, \QQ/\ZZ) \oplus \QQ/\ZZ$$
  sitting in a commutative diagram
  \begin{center}
    \begin{tikzcd}
      \Br(E_q\tensor L) \arrow{r}{\xi_q}[swap]{\approx}& H^1_q(G_L, \QQ/\ZZ) \oplus \QQ/\ZZ\\
      \Br(E_q) \arrow{u}{\res}\arrow{r}{\xi_q}[swap]{\approx}& H^1_q(G_K, \QQ/\ZZ) \oplus \QQ/\ZZ \arrow{u}[swap]{\res\oplus [L:K]}
    \end{tikzcd}
  \end{center}
\end{lemma}

\begin{proof}
  First, observe that the continuous map
  $\theta_L \colon L^\unit \into G_L^{\rm ab}$ induces an isomorphism
  $$H^1(G_L, \QQ/\ZZ) = \Hom_c(G_L^{\rm ab}, \QQ/\ZZ) \stackrel{\theta_L^\vee}{\to} \Hom_c(L^\unit, \QQ/\ZZ)
  $$
  In fact, $\theta_L^\vee$ is injective since $\theta_L$ has dense
  image.  To show that $\theta_L^\vee$ is surjective, let
  $f\in \Hom_c(L^\unit, \QQ/\ZZ)$, and let $\pi_L$ be a uniformizer of
  the ring of integers $R_L$ of $L$, so that
  $L^\unit \cong \pi_L^\ZZ \oplus U_L$ where $U_L$ denotes the group
  of units of $R_L$.  Since $\QQ/\ZZ$ has discrete topology, $\ker f$
  must be open in $L^\unit$, i.e.,
  $\ker f \supseteq U^{(i)}_L \df= 1 + (\pi_L)^i$ for some $i\ge 1$.
  Since $L^\unit \cong \pi_L^\ZZ \oplus U_L$, and $f(\pi_L)$ and
  $U_L/U^{(i)}_L$ have finite orders, $\im f$ is finite, and thus
  $\ker f$ has finite index in $L^\unit$, therefore
  $\ker f = N_{M/L} (M^\unit)$ is a norm group for some finite abelian
  extension $M\supseteq L$ by the existence theorem (\cite{SerreLF},
  XIV.\S6, theorem~1, p.218).  Therefore $f$ will be the image under
  $\theta_L^\vee$ of the composition
  \begin{center}
    \begin{tikzcd}
      G_L^{\rm ab} \arrow[two heads]{r}
      & \Gal(M/L) \arrow{r}{\theta_{M/L}^{-1}}[swap]{\approx}
      & \dfrac{L^\unit}{N_{M/L} (M^\unit)} = \dfrac{L^\unit}{\ker f} \arrow{r}{f}
      & \QQ/\ZZ
    \end{tikzcd}
  \end{center}
  where $\theta_{M/L}$ denotes the isomorphism induced on finite
  quotients by $\theta_L$.

  Next, by theorem~\ref{thm:tateellcurve}(i), there is an exact
  sequence of topological groups
  $$0 \to \ZZ \to L^\unit \to E_q(L) \to 0$$
  which, together with the isomorphism
  $\theta_L^\vee \colon H^1(G_L, \QQ/\ZZ) \isoto \Hom_c(L^\unit,
  \QQ/\ZZ)$, gives rise to an exact sequence
  $$0 \to \Hom_c(E_q(L), \QQ/\ZZ) \to H^1(G_L, \QQ/\ZZ) \stackrel{\epsilon_q}{\to} \Hom_c(\ZZ,\QQ/\ZZ)=\QQ/\ZZ
  $$
  Hence $E_q(L)^\vee = \ker \epsilon_q = H^1_q(G_L, \QQ/\ZZ)$.
  Therefore, by Lichtembaum's duality
  (theorem~\ref{thm:lichtembaums-duality}), we get isomorphisms
  $$\Br(E_q\tensor L) \stackrel{\textrm{\ref{thm:lichtembaums-duality}}}{=} \Pic(E_q\tensor L)^\vee
  = (E_q(L)\oplus\ZZ)^\vee = H^1_q(G_L, \QQ/\ZZ) \oplus \QQ/\ZZ
  $$
  Here we have used the identifications
  \begin{align*}
    E_q(L)&\isoto \Pic^0(E_q \tensor_K L)& \Pic(E_q\tensor_K L)&\isoto  \Pic^0(E_q\tensor_K L) \oplus \ZZ\\
          P&\mapsto [P]-[O_L]&                            \alpha&\mapsto (\alpha - (\deg \alpha)\cdot [O_L], \deg \alpha)
  \end{align*}
  where $[P]\in \Pic(E_q\tensor_K L)$ denotes the class of the
  degree~$1$ divisor given by $P\in E_q(L)$, and $O_L\in E_q(L)$ is
  the identity element of the elliptic curve.

  Finally, to prove that the diagram in the statement of the lemma
  commutes, write $f \colon E_q\tensor_K L \to E_q$ for the natural
  map, let $n = [L:K]$, and
  $\sigma_1, \ldots, \sigma_n\colon L \into \overline L$ be the
  $K$-imersions of $L$ into its separable closure
  $\overline L = \overline K$.  Then
  $f_*\colon \Pic(E_q\tensor_K L) \to \Pic(E_q)$ is given on closed
  points $P\in E_q(L)$ by
  $$f_* [P] = \sum_{1\le i\le n} [\sigma_i(P)]
  \in \Pic(E_q\tensor_K \overline K)^{G_K} = \Pic(E_q)
  $$
  where $\Pic(E_q\tensor_K \overline K)^{G_K} = \Pic(E_q)$ follows
  from the Hochschild-Serre spectral sequence, and the fact that $E_q$
  has a $K$-rational point, so that $\Br(K) \to \Br(E_q)$ is injective
  (see for example \cite{Lichtenbaum}, \S2, p.122).  Hence
  $f_*[O_L] = n[O_K]$, and there is a commutative diagram
  \begin{center}
    \begin{tikzcd}[column sep=small]
      \Pic(E_q\tensor_K L)\arrow{r}{\approx} \arrow{d}[swap]{f_*}
      &\Pic^0(E_q\tensor_K L) \oplus \ZZ \arrow[equal]{r} \arrow{d}[swap]{f_*}
      &E_q(L) \oplus \ZZ \arrow{r}{\approx} \arrow{d}{f_*}
      & L^\unit/q^\ZZ \oplus \ZZ  \arrow{d}{N_{L/K}\oplus n}\\
      \Pic(E_q) \arrow{r}{\approx}
      & \Pic^0(E_q) \oplus \ZZ\arrow[equal]{r}
      & E_q(K) \oplus \ZZ \arrow{r}{\approx}
      & K^\unit/q^\ZZ \oplus \ZZ
    \end{tikzcd}
  \end{center}
  where the rightmost vertical arrow is induced by the norm on the
  first component, and by multiplication by $n = [L:K]$ on the second.
  Combining the above diagram with the next one, expressing one of the
  functorial properties of the local reciprocity map,
  \begin{center}
    \begin{tikzcd}
      L^\unit \arrow{d}[swap]{N_{L/K}} \arrow{r}{\theta_L}& G_L^{\rm ab}\arrow[hook]{d}\\
      K^\unit \arrow{r}{\theta_K}& G_K^{\rm ab}
    \end{tikzcd}
  \end{center}
  taking Pontryagin duals, and applying Lichtenbaum's duality finishes
  the proof.
\end{proof}

The following lemma is well-known, but we include its short proof for
the convenience of the reader.

\begin{lemma}
  Let $m\in \NN$ be such that $\gcd(m, |\mu(K)|) = 1$ where $\mu(K)$
  denotes the (finite) group of all roots of unity in $K$.  Then the
  character group $H^1(G_K, \QQ/\ZZ)$ is $m$-divisible.
\end{lemma}

\begin{proof}
  Taking $G_K$-invariants of
  $0 \to (\tfrac{1}{m}\ZZ)/\ZZ \to \QQ/\ZZ \stackrel{m}{\to} \QQ/\ZZ
  \to 0$, we obtain an exact sequence
  $$H^1(G_K, \QQ/\ZZ) \stackrel{m}{\to} H^1(G_K, \QQ/\ZZ) \to H^2(G_K, (\tfrac{1}{m}\ZZ)/\ZZ)=0$$
  whose last term vanishes by local duality (\cite{NeukirchCNF},
  theorem~7.2.6, p.327):
  $$H^2(G_K, (\tfrac{1}{m}\ZZ)/\ZZ) = H^0(G_K, \mu_m)^\vee = \mu_m(K)^\vee =0
  $$
\end{proof}

Now we can show the main result in this section.

\begin{theorem}\label{thm:tatecurves}
  Let $m\in \NN$, and $E_q$ be a Tate elliptic curve over $K$.
  Suppose that ($p\mid m$ is allowed)
  \begin{enumerate}[(i)]
  \item either $m$ is prime; or
  \item $\gcd(m, |\mu(K)|) = 1$.
  \end{enumerate}
  Then all elements of $\Br(E_q)[m]$ are $\ZZ/m$-cyclic.
\end{theorem}

\begin{proof}
  By looking at $\ell$-primary parts, it suffices to consider the case
  when $m = \ell^r$ is a power of some prime $\ell$.  Let
  $\beta \in \Br(E_q)[m]$ be an element of order $m$, and
  $\xi_q(\beta) = (\chi, a) \in H^1_q(G_K, \QQ/\ZZ) \oplus \QQ/\ZZ$.
  Then either $\chi$ or $a$ has order $m = \ell^r$.  If
  $\chi\in H^1_q(G_K, \QQ/\ZZ)$ has order $m$, then it corresponds to
  a surjective morphism $\chi\colon G_K \onto \ZZ/m\ZZ$, hence to a
  degree $m$ cyclic extension $L \supseteq K$.  By the commutative
  diagram in lemma~\ref{lemma:brgptatecurve}, $\beta|_L = 0$, i.e.,
  $\beta$ is $\ZZ/m$-cyclic.  If $\chi$ has order $\ell^s$ with
  $s< r$, then there exists a character
  $\tilde \chi \in H^1(G_K, \QQ/\ZZ)$ of order $m = \ell^r$ such that
  $\ell^{r-s} \tilde \chi = \chi$; this is clear in case (i), while in
  case (ii) it follows from the previous lemma.  Then $\tilde \chi$
  defines a degree $\ell^r$ cyclic extension $L\supseteq K$ such that
  $\beta|_L = 0$, again by the diagram in
  lemma~\ref{lemma:brgptatecurve}.
\end{proof}

Noting that $|\mu(\QQ_p)| = p-1$ we obtain the following interesting

\begin{corollary}
  Let $E_q$ be a Tate elliptic curve over $\QQ_p$, and $r\in \NN$.
  Then all classes in $\Br(E_q)[p^r]$ are $\ZZ/p^r$-cyclic.
\end{corollary}



\section{Indecomposable algebras}

In this section, we assume once again that $X$ is the generic fiber of
a smooth projective $R$-curve $\XX$.  Our goal is to show that there
are indecomposable algebras in the function field $K(X)$ of $X$ with
period $p^2$ and index $p^3$, assuming enough roots of unity
(theorem~\ref{thm:indec-algebr}).  The construction basically follows
the strategy in \cite{BMT}, but we now have to struggle with the fact
that for the $p$-primary case we do not have at our disposal general
lifting theorems, as was the case in the prime-to-$p$ situation
treated in that paper.  That unfortunately restricts our construction
to the above period $p^2$ and index $p^3$.  As in \cite{BMT}, we rely
on the following well-known criterion for indecomposability:

\begin{lemma}\label{lemma:saltman-trick}
  Let $n = p^r$ and $\beta \in \Br(K(X))[n]$ be such that
  $\ind (p\beta) = \frac1p \ind \beta$.  Then $\beta$ is
  indecomposable.
\end{lemma}

We begin with a

\begin{lemma}
  Let $C$ be a smooth projective curve over the finite field
  $k = \FF_q$.  For all sufficiently large integers $m\gg 0$, there
  exist distinct closed points $P_\infty, P_1, P_2\in C$, all of
  degree $m$, such that the divisors $[P_1]-[P_\infty]$ and
  $[P_2]-[P_\infty]$ are principal.
\end{lemma}

\begin{proof}
  Let $g$ be the genus of $C$, and for each $m\in \NN_{>0}$ denote by
  $S(m)$ the set of all closed points $P \in C$ of degree
  $[\kappa(P): k] = m$.  By \cite{HKT}, theorem~9.25, p.346,
  $$\left||S(m)|-\frac{q^{m}}{m}\right| <(2+7 g) \cdot
  \frac{\sqrt{q^{m}}}{m} \text{ for all } m\ge 2
  \implies \lim_{m\to \infty} \frac{|S(m)|}{q^m/m} = 1
  $$
  Hence $|S(m)|$ asympotically grows as $q^m/m$, and if $m$ is large
  enough then $|S(m)| > 2 |\Pic^0(C)|$.  Fix such an $m$, and pick an
  arbitrary point $Q \in S(m)$.  Then in the set
  $\{ [P] - [Q] \in \Div^0(C) \mid P \in S(m) \}$ there will be three
  distinct divisors $[P_\infty] - [Q]$, $[P_1] - [Q]$, $[P_2] - [Q]$
  having the same image in $\Pic^0(C)$, hence their differences
  $[P_1]-[P_\infty]$ and $[P_2]-[P_\infty]$ will have trivial image in
  $\Pic^0(C)$, i.e., they will be principal divisors.
\end{proof}

Let $\XX\to \Spec R$ be a smooth projective curve as above.  Fix
$m\in \NN$ such that
\begin{itemize}
\item $m$ is not divisible by $p$;
\item $m > 2g-2$, where $g$ is the genus of $\XX_s$;
\item $m\gg 0$ satisfies the conditions of the lemma for $C = \XX_s$,
  so that there are three distinct degree $m$ points
  $P_\infty, P_1, P_2 \in \XX_s$ such that
  $$\divisor(\overline f_i) = [P_i] - [P_\infty] \in \Div(\XX_s)
  \qquad(i=1,2)
  $$
  for some elements $\overline f_i \in k(\XX_s)^\unit$ in the function
  field of $\XX_s$.
\end{itemize}
By lemma~\ref{lemma:integral-models}(iii), there is a horizontal
divisor $D_\infty$ on $\XX$ intersecting the special fiber $\XX_s$
transversally at $P_\infty\in \XX_s$.  And since $m > 2g-2$,
\cite{HarbaterHartmann} proposition~4.1, p.71, shows that the map
$H^0(\XX, \sheaf{O}_{\XX}(D_\infty)) \onto H^0(\XX_s,
\sheaf{O}_{\XX_s}(P_\infty))$ is surjective, hence we may choose
$f_i\in H^0(\XX, \sheaf{O}_{\XX}(D_\infty)) \subset K(X)$ lifting
$\overline f_i$.  Then all components of $\divisor(f_i)$ are
horizontal (since $\overline f_i \ne 0$), and $D_\infty$ is the only
pole of $f_i$.  Hence
$$\divisor(f_i) = [D_i] - [D_\infty]\in \Div(\XX)
\qquad(i=1,2)
$$
where $D_i$ is a horizontal divisor on $\XX$ restricting
(transversally) to $P_i\in \XX_s$ for $i=1,2$.  Moreover, since
$\sheaf{O}_{\XX_s, P_i}/(\overline f_i) \supset k$ is a (separable)
degree $m$ field extension, $\sheaf{O}_{\XX, P_i}/(f_i) \supset R$ is
a finite unramified extension of degree $m$, hence
$\sheaf{O}_{\XX, P_i}/(f_i)$ is a complete $p$-adic ring with
uniformizer $\pi$.  Now set $f = f_1f_2$ so that it restricts to
$\overline f = \overline f_1 \overline f_2$ and
\begin{align*}
  \divisor(f) &= [D_1] + [D_2] - 2[D_\infty] \in \Div(\XX)\\
  \divisor(\overline f) &= [P_1] + [P_2] - 2[P_\infty] \in \Div(\XX_s)
\end{align*}

From now on, suppose that $p\ne 2$ and $\mu_{p^2} \subset K$.  We keep
the notation of subsection~\ref{sec:kumm-artin-schr}, writing
$\zeta_2\in \mu_{p^2}$ for a primitive $p^2$-th root of unity,
$\lambda_2 = \zeta_2-1$, and so on.  Observe that $\lambda_2$ is a
uniformizer of $\QQ_p(\zeta_2)$, a subfield of $K$, hence
$\pi\mid \lambda_2$ in $R$.  Consider the degree $p^2$
Artin-Schreier-Witt extension of $k(\XX_s)$ given by
\begin{align}
  x^p - x &= \overline f \label{eq:5}\\
  y^p - y - c(x^p, - x)  &= \overline f^{-1}\label{eq:6}
\end{align}
where $c(X, Y)\in \ZZ[X, Y]$ is the polynomial \eqref{eq:10} defined
in subsection~\ref{sec:kumm-artin-schr}.  Since $p\ne 2$, it is easy
to check that both $P_1$ and $P_2$ split completely in
$k(\XX_s)(x) \supset k(\XX_s)$, and each place above $P_1$ or $P_2$ is
totally ramified in $k(\XX_s)(x, y) \supset k(\XX_s)(x)$, while
$P_\infty$ is totally ramified in $k(\XX_s)(x, y) \supset k(\XX_s)$.
In particular, this shows that the latter extension indeed has degree
$p^2$.

Now we lift $k(\XX_s)(x, y) \supset k(\XX_s)$ to the Kummer extension
of $K(X)$ defined by equations
\begin{align}
  \frac{\left(\lambda x+1\right)^{p}-1}{\lambda^{p}} &= f\label{eq:4}\\
  \frac{1}{\lambda^{p}}\left[\frac{\left(\lambda
  y+F(x)\right)^{p}}{\lambda x+1}-G(f)\right]& = \frac{1}{f}\label{eq:7}
\end{align}
where $F, G\in R[x]$ are the $p$-truncated exponentials \eqref{eq:11}.
Let $\chi\in H^1(K(X), \ZZ/p^2)$ be a Kummer character of order $p^2$
associated to this extension, and $W$ (respectively $Y$) be the
normalization of $X$ in the Kummer extension of $K(X)$ defined by
\eqref{eq:4} (respectively \eqref{eq:4} and \eqref{eq:7}), so that the
composition
$$Y \to W \to X$$
defines a degree $p^2$ cyclic cover of $X$.  Observe that the function
fields of $W$ and $Y$ are given by $K(W) = K(X)(x)$ and
$K(Y) = K(X)(x, y)$, respectively.

\begin{lemma}\label{lemma:indec-algebr}
  In the above notation, for $i=1,2$ let $A_i$ and $B_i$ be the
  normalizations of $\sheaf{O}_{\XX, P_i}$ in $K(W)$ and $K(Y)$
  respectively.  Then
  \begin{enumerate}[(i)]
  \item $A_i$ and $B_i$ are regular semilocal rings (hence UFDs), each
    with exactly $p$ maximal ideals.
    
  \item The prime divisor $D_i$ of $\XX$ splits completely into $p$
    disjoint horizontal divisors $D_{i1}', \ldots, D_{ip}'$ in
    $\Spec A_i$, i.e.,
    $$D_i \times_{\XX} \Spec A_i = D_{i1}' \sqcup \cdots \sqcup D_{ip}'$$
    with each $D_{ij}'\cong D_i$, corresponding bijectively to the
    maximal ideals of $A_i$.  In particular, $f$ is a local parameter
    for each $D_{ij}'$.

  \item Each $D_{ij}'$ is totally ramified in $\Spec B_i$.
  \end{enumerate}
\end{lemma}

\begin{proof}
  We will prove the lemma for $i=1$.  Observe that $P_1$ is the closed
  point of $D_1$, $f_2\in \sheaf{O}_{\XX, P_1}^\unit$ (since
  $\overline f_2(P_1) \ne 0$), both $f_1$ and $f = f_1f_2$ are local
  parameters for $D_1$ in $\Spec \sheaf{O}_{\XX, P_1}$, and
  $(\pi, f) = (\pi, f_1)$ is the maximal ideal of the $2$-dimensional
  regular local ring $\sheaf{O}_{\XX, P_1}$.

  Expanding \eqref{eq:4}, since $\lambda^{p-1} \mid p$ in
  $R \subset \sheaf{O}_{\XX, P_1}$, we get the following monic
  equation for $x$ with coefficients in $\sheaf{O}_{\XX, P_1}$
  \begin{equation}\label{eq:12}
    x^p + \frac{p}{\lambda} x^{p-1} + \frac{p(p-1)}{2\lambda} x^{p-2} + \cdots + \frac{p}{\lambda^{p-1}}x - f = 0
  \end{equation}
  whose reduction modulo $\pi$ is \eqref{eq:5}.  Since
  $\sheaf{O}_{\XX, P_1}[x]\supset \sheaf{O}_{\XX, P_1}$ is an integral
  extension, if we show that $\sheaf{O}_{\XX, P_1}[x]$ is a regular
  semilocal ring then $\sheaf{O}_{\XX, P_1}[x]$ will be a UFD (by
  Auslander-Buchsbaum's theorem and \cite{Matsumura}, exercise 20.5,
  p.169), hence normal, showing also that
  $A_1 = \sheaf{O}_{\XX, P_1}[x]$.  Since \eqref{eq:12} reduces to
  $x^p-x = 0$ modulo the maximal ideal $(\pi, f)$ of
  $\sheaf{O}_{\XX, P_1}$, the fiber of $(\pi, f)$ under
  $\Spec \sheaf{O}_{\XX, P_1}[x] \to \Spec \sheaf{O}_{\XX, P_1}$ is
  given by
  $\Spec \sheaf{O}_{\XX, P_1}[x]/(\pi, f) = \Spec
  \kappa(P_1)[x]/(x^p-x)$, a disjoint union of $p$ copies of
  $\Spec \kappa(P_1)$.  This shows that $\sheaf{O}_{\XX, P_1}[x]$ is
  semilocal with $p$ maximal ideals $(x - a, \pi, f)$ for
  $a=0, 1, \ldots, p-1$.  Equation \eqref{eq:12} shows that
  $(x - a, \pi, f) = (x-a, \pi)$ in the localization
  $\sheaf{O}_{\XX, P_1}[x]_{(x - a, \pi, f)}$, therefore
  $\sheaf{O}_{\XX, P_1}[x]$ is regular.

  Since
  $D_1 \times_{\XX} \Spec A_1 = \Spec \sheaf{O}_{\XX, P_1}[x]/(f)$, in
  order to show (ii) we have to show that
  $\sheaf{O}_{\XX, P_1}[x]/(f)$ is isomorphic to a product of $p$
  copies of $\sheaf{O}_{\XX, P_1}$.  But $\sheaf{O}_{\XX, P_1}/(f)$ is
  a complete $p$-adic ring with uniformizer $\pi$, and \eqref{eq:12}
  reduces to the separable equation $x^p-x =0$ modulo $(\pi, f)$.
  Therefore the result follows by Hensel's lemma.
  
  Expanding \eqref{eq:7} and multiplying by $(1+\lambda x) f/y^p$ we get
  \begin{gather}
    \left(\frac{F(x)^p - (1+\lambda x) G(f)}{\lambda^p}\cdot f - (1+\lambda x) \right)\cdot (1/y)^p
    + \frac{pf}{\lambda^{p-1}} F(x)^{p-1} \cdot (1/y)^{p-1}\nonumber\\
    \qquad\qquad + \cdots + \frac{pf}{\lambda} F(x)\cdot (1/y) + f = 0\label{eq:8}
  \end{gather}
  The last equation implies that $A_1[1/y] \supset A_1$ is an integral
  extension.  In fact, by \cite{SekiguchiSuwa}, 5.15, p.236, one has
  that
  $$F(T)^{p} \equiv (\lambda T + 1)\cdot G\left(\frac{(\lambda T + 1)^p-1}{\lambda^p}\right) \pmod{\lambda^p}$$
  Moreover, $1+\lambda x$ is a unit in the semilocal ring
  $A_1 = \sheaf{O}_{\XX, P_1}[x]$ since $\lambda$ belongs to the
  maximal ideal $(\pi, f)$ of $\sheaf{O}_{\XX, P_1}$.  Therefore the
  coefficient of $(1/y)^p$ belongs to $A_1^\unit$.  Equation
  \eqref{eq:8} also shows that the fiber under
  $\Spec A_1[1/y] \to \Spec A_1$ of any maximal ideal $(x-a, \pi, f)$
  with $a=0, 1, \ldots, p-1$ is
  $\Spec \frac{\kappa(P_1)[1/y]}{(1/y)^p}$, i.e., consists of a single
  maximal ideal $(x-a, \pi, f, 1/y)$.  Hence $A_1[1/y]$ is semilocal
  with $p$ maximal ideals.

  Next we show that the $2$-dimensional semilocal ring $A_1[1/y]$ is
  regular, hence normal, thus proving that $B_1 = A_1[1/y]$.  In the
  localization ${A_1}_{(x-a, \pi, f)}$, we have
  $(x-a, \pi, f) = (\pi, f)$ by \eqref{eq:12}, hence in
  $A_1[1/y]_{(x-a, \pi, f, 1/y)}$ we have
  $(x-a, \pi, f, 1/y) = (1/y, \pi)$ by \eqref{eq:8}, showing that
  $A_1[1/y]$ is regular.

  Finally, if $D_{1i}'\subset \Spec A_1$ is the prime divisor
  corresponding to the maximal ideal $(x-a, \pi, f)$ of $A_1$ in (ii),
  since $f$ is a local parameter for
  $D_{1i}' = \Spec {A_1}_{(x-a, \pi, f)}/(f)$, from \eqref{eq:8} we
  get
  $$D_{1i}' \times_{\Spec A_1} \Spec B_1 = \Spec \frac{{A_1}_{(x-a, \pi, f)}[1/y]}{(f, (1/y)^p)}
  $$
  Hence the pre-image of $D_{1i}'$ in $\Spec B_1$ is the divisor
  locally given by $1/y$, with multiplicity $p$, i.e., $D_{1i}'$ is
  totally ramified in $\Spec B_1$.
\end{proof}

Now we can show

\begin{theorem}\label{thm:indec-algebr}
  Suppose $p\ne 2$, and let $X$ be the generic fiber of a smooth
  projective geometrically connected $R$-curve $\XX$.  Assume
  $\mu_{p^2} \subset K$, and let $\chi\in H^1(K(X), \ZZ/p^2)$ be the
  degree $p^2$ character constructed above, corresponding to the field
  extension $K(Y) \supset K(X)$.  Let $h = f_1/f_2$ and
  $\psi\in H^1(K, \ZZ/p^2)$ be the $\pi$-unramified character of order
  $p^2$.  Define
  $$\beta \df= \psi \cup \delta_{p^2} (h) + \chi \cup \delta_{p^2} (\pi)\in \Br(K(X))[p^2]$$
  Then $\beta$ is the class of an indecomposable algebra with
  $\ind \beta = p^3$.
\end{theorem}

The proof of the theorem will immediately follow from
lemma~\ref{lemma:saltman-trick} once we show the next two lemmas.

\begin{lemma}
  $p^3\mid \ind\beta$ and $p^2\mid \ind (p\beta)$.
\end{lemma}

\begin{proof}
  Let $\widehat{K(X)}$ be the completion of $K(X)$ with respect to the
  discrete valuation defined by the special fiber $\XX_s$.  Notice
  that $\widehat{K(X)}$ has uniformizer $\pi$ and residue field
  $k(\XX_s)$ (the function field of $\XX_s$), a global field of
  characteristic $p$.  Our goal is to show that the restriction of
  $\beta$ to $\widehat{K(X)}$ has index $p^3$.

  Since the compositum of $\widehat{K(X)}(\psi)$ and
  $\widehat{K(X)}(\chi)$ splits $\beta$, and both define
  $\pi$-unramified extensions of $\widehat{K(X)}$, we may apply
  lemma~\ref{lemma:NAK} in order to compute
  $$\ind \beta|_{\widehat{K(X)}}
  = |\chi| \cdot \ind (\psi\cup \delta h)|_{k(\XX_s)(\chi)} = p^2
  \cdot \ind (\psi\cup \delta h)|_{k(\XX_s)(\chi)}
  $$
  Let $\alpha = (\psi\cup \delta h)|_{k(\XX_s)} \in \Br(k(\XX_s))$.
  We must show that $\ind \alpha|_{k(\XX_s)(\chi)} = p$, and for that
  we need to compute the local invariants of
  $\alpha|_{k(\XX_s)(\chi)}$.  Since $\psi$ defines an unramified
  cover of $\XX_s$, and
  $\divisor(\overline h) = [P_1] - [P_2]\in \Div(\XX_s)$ with
  $m = [\kappa(P_i): k]$ prime to $p$ (where
  $\overline h = \overline f_1/\overline f_2$), $\alpha$ has
  nontrivial invariants $\pm 1/p^2\in \QQ/\ZZ$ only at $P_1$ and
  $P_2$.  On the other hand, by construction
  $k(\XX_s)(\chi)\supset k(\XX_s)$ is the Artin-Schreier-Witt
  extension given by equations \eqref{eq:5} and \eqref{eq:6}, which
  defines local extensions of degree $p$ at the completions of
  $k(\XX_s)$ at $P_1$ and $P_2$.  Therefore $\alpha|_{k(\XX_s)(\chi)}$
  will have local invariants $\pm 1/p\in \QQ/\ZZ$ at those places, and
  thus its index is $p$, as desired.  A similar argument shows that
  $\ind (p\beta)|_{\widehat{K(X)}} = p^2$.
\end{proof}

For the next and final lemma, we make some preliminary remarks.
First, for any $x\in X_0$ denote by $R_x$ the ring of integers of the
$p$-adic field $\kappa(x)$.  Then by the valuative criterion of
properness the inclusion $x = \Spec \kappa(x) \into X \subset \XX$
extends uniquely to an $R_x$-valued point $\Spec R_x \into \XX$ (the
normalization of the horizontal divisor
$\overline{\{ x \}}\subset \XX$, see
lemma~\ref{lemma:integral-models}(ii)).

Second, let $M\supset K(X)$ be a field extension, let $Z$ be the
$p$-adic curve obtained by normalizing $X$ in $M$, and $Z \to X$ be
the natural map induced by the inclusion of function fields
$K(X) \into M = K(Z)$.  By corollary~\ref{thm:hasse}, to show that $M$
splits some $\gamma\in \Br(K(X))$, we can proceed locally, showing
that for every closed point $z\in Z_0$ with image $x\in X_0$ the
completion $\widehat{K(Z)}_z$ of $K(Z)$ at $z$ splits
$\gamma|_{\widehat{K(X)}_x}$.

Third, let $x\in X_0$ and $z\in Z_0$ be a closed point lying over $x$
(with respect to the map $Z\to X$).  Let
$D_x \df= \overline{\{ x \}}\subset \XX$ be horizontal divisor
associated to $x$ (see lemma~\ref{lemma:integral-models}(ii)), write
$P_x\in \XX_s$ for the closed point of $D_x$, let
$A = \sheaf{O}_{\XX, P_x}$ and $\ideal{p}\in\Spec A$ be the height~$1$
prime ideal which is the generic point of $D_x = \Spec A/\ideal{p}$.
Let $C\subset K(Z)$ be the normalization of $A$ in $K(Z)$.  We claim
that $\sheaf{O}_{Z,z} = C_{\ideal{q}}$ for some height~$1$ prime ideal
$\ideal{q}\in \Spec C$, which is contained in a single maximal ideal
of $C$.  In fact, since $\sheaf{O}_{Z,z}$ is a localization of the
normalization of $\sheaf{O}_{X,x}$ in $K(Z)$, we have the following
inclusions of rings, and corresponding maps of spectra
\begin{center}
  \begin{tikzcd}
    C\arrow[hook]{r} &  \sheaf{O}_{Z,z}\\
    A\arrow[hook]{r} \arrow[hook]{u} &  \sheaf{O}_{X,x} \arrow[hook]{u} \arrow[equal]{r}& A_{\ideal{p}}
  \end{tikzcd}
  \qquad\qquad
  \begin{tikzcd}
\ideal{q}\arrow[mapsto]{d}& \Spec C \arrow{d} &  \Spec \sheaf{O}_{Z,z}\arrow{d}\arrow{l}\\
\ideal{p}&\Spec A &  \Spec \sheaf{O}_{X,x}  \arrow{l}
  \end{tikzcd}
\end{center}
Let $\ideal{q}\in\Spec C$ be the image of the closed point of
$\Spec \sheaf{O}_{Z, z}$, so that $\ideal{q}$ lies over $\ideal{p}$.
Since $\ideal{p}$ has height~$1$ and $A \into C$ is integral with $A$
normal, by the going-up and going-down theorems $\ideal{q}$ also has
height $1$, thus $C_{\ideal{q}}$ is a discrete valuation ring (being a
local noetherian normal domain of $\dim C_{\ideal{q}}=1$).  Then the
inclusion $C\into \sheaf{O}_{Z,z}$ induces a local injective map
$C_{\ideal{q}} \into \sheaf{O}_{Z,z}$, which must be an equality since
discrete valuation rings are maximal with respect to domination of
valuation rings (\cite{Matsumura}, exercise 10.5, p.77).  Finally, we
have to check that $\ideal{q}$ is contained in a single maximal ideal
of $C$, i.e., that $C/\ideal{q}$ is a local domain.  Consider the
inclusion $A/\ideal{p} \into C/\ideal{q}$, a finite extension of
domains.  Since $D_x = \Spec A/\ideal{p}$ is a horizontal divisor,
$\Spec A/\ideal{p}$ is finite (and flat) over $\Spec R$
(lemma~\ref{lemma:integral-models}(i)), hence
$\Spec C/\ideal{q}\to \Spec R$ is finite with $R$ henselian.  Now the
result follows from \cite{Milne}, theorem~I.4.2(b), p.32.

\begin{lemma}
  Let $L\supset K$ be the $\pi$-unramified extension of degree $p$.
  Let $M = L(Y)$ be the compositum of $L$ and $K(Y)$.  Then
  $[M:K(X)] = p^3$ and $M$ splits $\beta$, thus $\ind\beta$ is at most
  $p^3$.  Similarly $\ind (p\beta)$ is at most $p^2$.
\end{lemma}

\begin{proof}
  First observe that $L$ and $K(Y)$ are linearly disjoint over $K(X)$
  (for instance, by lemma~\ref{lemma:indec-algebr} the maximal ideals
  of $\Spec B_1$ lying over $P_1$ define trivial residue field
  extensions in $\Spec B_1\to \Spec \sheaf{O}_{\XX, P_1}$).  Hence
  $M = L(Y)\supset K(X)$ is a degree $p^3$ extension, and since it
  visibly splits $\chi\cup \delta \pi$, we must show that it splits
  the remaining part $\gamma = \psi\cup \delta h$.

  Let $Z$ be the normalization of $X$ in $M$.  The inclusions
  $M \supset K(Y) \supset K(X)$ define maps of $p$-adic curves
  $Z \to Y \to X$.  Fix a closed point $z\in Z_0$, and let $y\in Y_0$
  and $x\in X_0$ be its images.  Our goal is to show that
  $\widehat{K(Z)}_z$ splits $\gamma|_{\widehat{K(X)}_x}$.  Let
  $P_x \in \XX_s$ be the closed point of the horizontal divisor
  $\overline{\{x\}}$.  There are a few cases to consider:
  \begin{enumerate}[(i)]
  \item $P_x\ne P_1, P_2$.  Then $h$ is a unit in
    $\sheaf{O}_{\XX, P_x}$, thus also in $\sheaf{O}_{X, x}$.  Since
    $\psi$ is also unramified with respect to $x$,
    $\gamma|_{\widehat{K(X)}_x}\in \Br(\kappa(x))$.  But
    $\gamma|_{\widehat{K(X)}_x} = (\psi \cup \delta h)|_{\kappa(x)}$
    is already trivial in $\Br(\kappa(x))$, since $\psi$ is also
    $\pi$-unramified and $h\in \sheaf{O}_{\XX, P_x}^\unit$ implies
    that the image of $h$ in $R_x$ (the valuation ring of $\kappa(x)$,
    a normalization of a quotient of $\sheaf{O}_{\XX, P_x}$) is also a
    unit.

  \item $P_x = P_1$ (or the analogous case $P_x = P_2$).  Write
    $B \subset K(Y)$ and $C\subset M$ for the normalizations of
    $A = \sheaf{O}_{\XX, P_1}$ in $K(Y)$ and $M = K(Z)$, respectively.
    By lemma~\ref{lemma:indec-algebr}, $B$ is a $2$-dimensional
    regular semilocal ring with $p$ maximal ideals
    $\ideal{m}_1, \ldots, \ideal{m}_p$.  Moreover, since both $f$ and
    $h$ are local parameters of $D_1$ in $\Spec A$, by the same lemma
    we may write $h = u t_1^p \ldots t_p^p$ where $u\in B^\unit$ and
    $t_i\in B$ are prime elements defining the irreducible components
    of $D_1 \times_{\XX} \Spec B$, with $t_i \in \ideal{m}_i$ and
    $t_i\notin \ideal{m}_j$ for $i\ne j$.  By the above remarks,
    $\sheaf{O}_{Y, y} = B_{\ideal{q}}$ for some height~$1$ prime ideal
    $\ideal{q} \in \Spec B$, say contained in $\ideal{m}_1$ (and not
    in any other $\ideal{m}_i$ with $i\ne 1$).  In $B_{\ideal{m}_1}$,
    $h = u't_1^p$ with $u' \in B_{\ideal{m}_1}^\unit$.  Therefore we
    may write
    $$\gamma|_M = (\psi \cup \delta h)|_M = (p\psi)|_M \cup \delta t_1 + (\psi\cup \delta u')|_M
    = (\psi\cup \delta u')|_M
    $$
    since $\psi|_L$ has order $p$ and thus $(p\psi)|_M = 0$.  But
    $u' \in B_{\ideal{m}_1}^\unit\implies u' \in \sheaf{O}_{Y,
      y}^\unit$, hence the same argument of (i) shows that
    $(\psi\cup \delta u')|_{\widehat{K(Y)}_y} = (\psi\cup \delta
    u')|_{\kappa(y)} = 0$.  Therefore
    $\gamma|_{\widehat{K(Z)}_z} = (\psi\cup \delta
    u')|_{\widehat{K(Z)}_z} = 0$ as well.
  \end{enumerate}
\end{proof}


\bibliographystyle{alpha} 
\bibliography{standardbib}

\def\cprime{$'$}
\begin{thebibliography}{SVdB92}

\bibitem[AG60]{AG}
Maurice Auslander and Oscar Goldman.
\newblock The {B}rauer group of a commutative ring.
\newblock {\em Trans. Amer. Math. Soc.}, 97:367--409, 1960.

\bibitem[ART79]{ART}
S.~A. Amitsur, L.~H. Rowen, and J.-P. Tignol.
\newblock Division algebras of degree {$4$} and {$8$} with involution.
\newblock {\em Israel J. Math.}, 33(2):133--148, 1979.

\bibitem[BMT11]{BMT}
E.~Brussel, K.~McKinnie, and E.~Tengan.
\newblock Indecomposable and noncrossed product division algebras over function
  fields of smooth {$p$}-adic curves.
\newblock {\em Adv. Math.}, 226(5):4316--4337, 2011.

\bibitem[BMT16]{BMT2}
Eric Brussel, Kelly McKinnie, and Eduardo Tengan.
\newblock Cyclic length in the tame {B}rauer group of the function field of a
  {$p$}-adic curve.
\newblock {\em Amer. J. Math.}, 138(2):251--286, 2016.

\bibitem[Bru96]{Brussel6}
Eric~S. Brussel.
\newblock Decomposability and embeddability of discretely {H}enselian division
  algebras.
\newblock {\em Israel J. Math.}, 96(, part A):141--183, 1996.

\bibitem[Bru10]{BrusselSaltman}
Eric Brussel.
\newblock On {S}altman's {$p$}-adic curves papers.
\newblock In {\em Quadratic forms, linear algebraic groups, and cohomology},
  volume~18 of {\em Dev. Math.}, pages 13--39. Springer, New York, 2010.

\bibitem[CS86]{CornellSilverman}
G.~Cornell and J.H. Silverman.
\newblock {\em Arithmetic geometry}.
\newblock Springer-Verlag, 1986.

\bibitem[Ful98]{Fulton}
William Fulton.
\newblock {\em Intersection theory}, volume~2 of {\em Ergebnisse der Mathematik
  und ihrer Grenzgebiete. 3. Folge. A Series of Modern Surveys in Mathematics
  [Results in Mathematics and Related Areas. 3rd Series. A Series of Modern
  Surveys in Mathematics]}.
\newblock Springer-Verlag, Berlin, second edition, 1998.

\bibitem[GMS03]{GMS}
Skip Garibaldi, Alexander Merkurjev, and Jean-Pierre Serre.
\newblock {\em Cohomological invariants in {G}alois cohomology}, volume~28 of
  {\em University Lecture Series}.
\newblock American Mathematical Society, Providence, RI, 2003.

\bibitem[Gro68]{GB}
Alexander Grothendieck.
\newblock Le groupe de {B}rauer. {III}. {E}xemples et compl\'ements.
\newblock In {\em Dix {E}xpos\'es sur la {C}ohomologie des {S}ch\'emas}, pages
  88--188. North-Holland, Amsterdam, 1968.

\bibitem[HH10]{HarbaterHartmann}
David Harbater and Julia Hartmann.
\newblock Patching over fields.
\newblock {\em Israel J. Math.}, 176:61--107, 2010.

\bibitem[HHK09]{HHK}
David Harbater, Julia Hartmann, and Daniel Krashen.
\newblock Applications of patching to quadratic forms and central simple
  algebras.
\newblock {\em Invent. Math.}, 178(2):231--263, 2009.

\bibitem[Hid12]{Hida}
Haruzo Hida.
\newblock {\em Geometric modular forms and elliptic curves}.
\newblock World Scientific Publishing Co. Pte. Ltd., Hackensack, NJ, second
  edition, 2012.

\bibitem[HKT08]{HKT}
J.~W.~P. Hirschfeld, G.~Korchm\'{a}ros, and F.~Torres.
\newblock {\em Algebraic curves over a finite field}.
\newblock Princeton Series in Applied Mathematics. Princeton University Press,
  Princeton, NJ, 2008.

\bibitem[Jac91]{J_2}
Bill Jacob.
\newblock Indecomposable division algebras of prime exponent.
\newblock {\em J. Reine Angew. Math.}, 413:181--197, 1991.

\bibitem[JW90]{JW}
Bill Jacob and Adrian Wadsworth.
\newblock Division algebras over {H}enselian fields.
\newblock {\em J. Algebra}, 128(1):126--179, 1990.

\bibitem[Kar98]{Karpenko}
Nikita~A. Karpenko.
\newblock Codimension {$2$} cycles on {S}everi-{B}rauer varieties.
\newblock {\em $K$-Theory}, 13(4):305--330, 1998.

\bibitem[Kat86]{KatoHP}
Kazuya Kato.
\newblock A {H}asse principle for two-dimensional global fields.
\newblock {\em J. Reine Angew. Math.}, 366:142--183, 1986.
\newblock With an appendix by Jean-Louis Colliot-Th\'el\`ene.

\bibitem[Lic69]{Lichtenbaum}
Stephen Lichtenbaum.
\newblock Duality theorems for curves over {$p$}-adic fields.
\newblock {\em Invent. Math.}, 7:120--136, 1969.

\bibitem[Lip78]{Lipman}
Joseph Lipman.
\newblock Desingularization of two-dimensional schemes.
\newblock {\em Annals of Mathematics}, 107(2):151--207, 1978.

\bibitem[Liu02]{Liu}
Qing Liu.
\newblock {\em Algebraic geometry and arithmetic curves}, volume~6 of {\em
  Oxford Graduate Texts in Mathematics}.
\newblock Oxford University Press, Oxford, 2002.
\newblock Translated from the French by Reinie Ern{\'e}, Oxford Science
  Publications.

\bibitem[Mat55]{Mattuck}
Arthur Mattuck.
\newblock Abelian varieties over p-adic ground fields.
\newblock {\em Annals of Mathematics}, 62(1):92--119, 1955.

\bibitem[Mat89]{Matsumura}
Hideyuki Matsumura.
\newblock {\em Commutative ring theory}, volume~8 of {\em Cambridge Studies in
  Advanced Mathematics}.
\newblock Cambridge University Press, Cambridge, second edition, 1989.
\newblock Translated from the Japanese by M. Reid.

\bibitem[McK08]{McKinnie}
Kelly McKinnie.
\newblock Indecomposable {$p$}-algebras and {G}alois subfields in generic
  abelian crossed products.
\newblock {\em J. Algebra}, 320(5):1887--1907, 2008.

\bibitem[Mil80]{Milne}
James~S. Milne.
\newblock {\em \'{E}tale cohomology}, volume~33 of {\em Princeton Mathematical
  Series}.
\newblock Princeton University Press, Princeton, N.J., 1980.

\bibitem[Mil86]{MilneAD}
J.~S. Milne.
\newblock {\em Arithmetic duality theorems}, volume~1 of {\em Perspectives in
  Mathematics}.
\newblock Academic Press, Inc., Boston, MA, 1986.

\bibitem[NS13]{Neukirch}
J.~Neukirch and N.~Schappacher.
\newblock {\em Algebraic Number Theory}.
\newblock Grundlehren der mathematischen Wissenschaften. Springer Berlin
  Heidelberg, 2013.

\bibitem[NSW00]{NeukirchCNF}
J{\"u}rgen Neukirch, Alexander Schmidt, and Kay Wingberg.
\newblock {\em Cohomology of number fields}, volume 323 of {\em Grundlehren der
  Mathematischen Wissenschaften [Fundamental Principles of Mathematical
  Sciences]}.
\newblock Springer-Verlag, Berlin, 2000.

\bibitem[Pop88]{Pop}
Florian Pop.
\newblock Galoissche kennzeichnung p-adisch abgeschlossener k{\"o}rper.
\newblock {\em Journal f{\"u}r die reine und angewandte Mathematik},
  392:145--175, 1988.

\bibitem[PS14]{ParimalaSuresh}
R.~Parimala and V.~Suresh.
\newblock Period-index and {$u$}-invariant questions for function fields over
  complete discretely valued fields.
\newblock {\em Invent. Math.}, 197(1):215--235, 2014.

\bibitem[Sai86]{Saito2}
Shuji Saito.
\newblock Arithmetic on two-dimensional local rings.
\newblock {\em Invent. Math.}, 85(2):379--414, 1986.

\bibitem[Sal79]{Saltman-Indecomposable}
David~J. Saltman.
\newblock Indecomposable division algebras.
\newblock {\em Comm. Algebra}, 7(8):791--817, 1979.

\bibitem[Sal97]{Saltman-p-adic}
David~J. Saltman.
\newblock Division algebras over {$p$}-adic curves.
\newblock {\em J. Ramanujan Math. Soc.}, 12(1):25--47, 1997.

\bibitem[Sal98]{Saltman-p-adic-c}
David~J. Saltman.
\newblock Correction to: ``{D}ivision algebras over {$p$}-adic curves'' [{J}.
  {R}amanujan {M}ath. {S}oc. {\bf 12} (1997), no. 1, 25--47; {MR}1462850
  (98d:16032)].
\newblock {\em J. Ramanujan Math. Soc.}, 13(2):125--129, 1998.

\bibitem[Sal07]{Saltman-cyclic}
David~J. Saltman.
\newblock Cyclic algebras over {$p$}-adic curves.
\newblock {\em J. Algebra}, 314(2):817--843, 2007.

\bibitem[Sal08]{Sa08}
David~J. Saltman.
\newblock Division algebras over surfaces.
\newblock {\em J. Algebra}, 320(4):1543--1585, 2008.

\bibitem[Ser79]{SerreLF}
Jean-Pierre Serre.
\newblock {\em Local fields}, volume~67 of {\em Graduate Texts in Mathematics}.
\newblock Springer-Verlag, New York, 1979.
\newblock Translated from the French by Marvin Jay Greenberg.

\bibitem[Sil94]{SilvermanATAEC}
J.H. Silverman.
\newblock {\em Advanced topics in the arithmetic of elliptic curves}.
\newblock Graduate texts in mathematics. Springer-Verlag, 1994.

\bibitem[SS01]{SekiguchiSuwa}
Tsutomu Sekiguchi and Noriyuki Suwa.
\newblock A note on extensions of algebraic and formal groups. {IV}.
  {K}ummer-{A}rtin-{S}chreier-{W}itt theory of degree {$p^2$}.
\newblock {\em Tohoku Math. J. (2)}, 53(2):203--240, 2001.

\bibitem[{Sta}18]{stacks-project}
The {Stacks Project Authors}.
\newblock \textit{Stacks Project}.
\newblock \url{https://stacks.math.columbia.edu}, 2018.

\bibitem[SVdB92]{SvdB}
Aidan Schofield and Michel Van~den Bergh.
\newblock The index of a {B}rauer class on a {B}rauer-{S}everi variety.
\newblock {\em Trans. Amer. Math. Soc.}, 333(2):729--739, 1992.

\bibitem[Tig87]{Tignol}
J.-P. Tignol.
\newblock Alg\`ebres ind\'ecomposables d'exposant premier.
\newblock {\em Adv. in Math.}, 65(3):205--228, 1987.

\end{thebibliography}


\end{document}